\documentclass{amsart}

\usepackage{amsmath}
\usepackage{amsthm}
\usepackage{amssymb}
\usepackage{mathrsfs}
\usepackage{xcolor}
\usepackage{hyperref}
\hypersetup{colorlinks=true,linktoc=section,urlcolor={blue!80!black},citecolor={green!60!black}, breaklinks=true}%%%%

\newcommand{\C}{\mathbb{C}}
\newcommand{\N}{\mathbb{N}}
\newcommand{\R}{\mathbb{R}}
\newcommand{\Z}{\mathbb{Z}}

\newcommand{\cA}{\mathcal{A}}
\newcommand{\cD}{\mathcal{D}}
\newcommand{\cF}{\mathcal{F}}

\newcommand{\e}{\varepsilon}

\newcommand{\wt}{\widetilde}

\newtheorem{thm}{Theorem}[section]

\newtheorem{lem}[thm]{Lemma}

\theoremstyle{definition}

\theoremstyle{remark}

\newtheorem{rem}[thm]{Remark}	

\begin{document}

\title[ Poisson Matrix $\mathbf{A}_2$ characteristic ]{The Poisson Matrix $\mathbf{A}_2$ characteristic and the 3/2 blow up of the Hilbert transform}

\author[Domelevo]{Komla Domelevo}
\address{Institut f{\"u}r Mathematik, Julius-Maximilians-Universit{\"a}t
W{\"u}rzburg, Emil-Fischer-Str.~41, 97074 W{\"u}rzburg, Germany}
\email{kolma.domelevo@uni-wuerzburg.de}

\author[Kakaroumpas]{Spyridon Kakaroumpas}
\address{Institut f{\"u}r Mathematik, Julius-Maximilians-Universit{\"a}t
W{\"u}rzburg, Emil-Fischer-Str.~41, 97074 W{\"u}rzburg, Germany}
\email{spyridon.kakaroumpas@uni-wuerzburg.de}

\author[Petermichl]{Stefanie Petermichl}
\address{Institut f{\"u}r Mathematik, Julius-Maximilians-Universit{\"a}t
W{\"u}rzburg, Emil-Fischer-Str.~41, 97074 W{\"u}rzburg, Germany}
\email{stefanie.petermichl@uni-wuerzburg.de}

\author[Treil]{Sergei Treil}
\address{Department of Mathematics, Brown University, Box 1917, 151 Thayer Street Providence, RI 02912, USA}
\email{treil@math.brown.edu}

\author[Volberg]{Alexander Volberg}
\address{Department of Mathematics, Michigan State University, 619 Red Cedar Road,
C212 Wells Hall, East Lansing, MI 48824, USA}
\email{volberg@msu.edu}

\begin{abstract}
    Recently the matrix $A_2$ conjecture was disproved. % by the first, the third, the fourth and the fifth authors.
    Indeed, the growth of the vector Hilbert transform in the matrix weighted $L^2(W)$ space was shown to be at best a constant multiple of $[W]_{\mathbf{A}_2}^{3/2}$. This bound had  previously been  established and it was thus proved  that it is sharp and the conjectured linear growth cannot be obtained. It is a natural question to see if the $3/2$ power persists if we replace the classical matrix $A_2$ characteristic by the ``fattened'', larger, so-called matrix Poisson $A_2$ characteristic. We show that the 3/2 power, even in this case, cannot be improved.
\end{abstract}

\maketitle

\section{Introduction}

The classical Hunt--Muckenhaupt--Wheeden theorem \cite{HuMuWh} characterized the boundedness of the Hilbert transform
$$\mathcal{H}f(x)=\frac1{\pi}\text{ p.v.}\int \frac{f(t)}{x-t}\,\mathrm{d}t$$
in weighted $L^2(w)$ spaces, where $w$ is an almost everywhere positive, locally integrable function.  The characterization was via the $A_2$ characteristic of the weight
\begin{equation}
    \label{eq:scalar_A2_characteristic}
    [w]_{{A}_2}=\sup_{I}\langle w \rangle_I \langle w^{-1} \rangle_I,    
\end{equation}
where the supremum runs over all intervals $I$ on the real line. The Hilbert transform maps thus boundedly in $L^2(w)$ if and only if the weight $w$ has a finite $A_2$ characteristic $[w]_{{A}_2}$.

This very characterization was a sought after but difficult task in case of a matrix weight $W$, now a locally integrable function with values in the symmetric positive definite matrices, and the vector Hilbert transform. The characterization was accomplished by the fourth and the fifth authors \cite{TrVo} in the late 90s, establishing the matrix $A_2$ characteristic
$$[W]_{\mathbf{A}_2}=\sup_{I}\|\langle W \rangle_I^{1/2} \langle W^{-1} \rangle_I^{1/2}\|^2,$$ where the supremum runs over all intervals on the real line. Again, the Hilbert transform acts boundedly in $L^2(W)$ if and only if the weight $W$ has a finite matrix $A_2$ characteristic $[W]_{\mathbf{A}_2}.$ 

There were numerous applications of such an estimate in probability and operator theory, see a summary in \cite{DoPeTrVo}. Let us also remark that it is known that the operator case fails \cite{GiPoTrVo} due to a blow up of the estimate with the side length of the matrix. Indeed, interest developed both in the exact dimensional growth as well as the growth with the $A_2$ characteristics. Both the scalar and the matrix $A_2$ characteristics have ranges of at least 1 due to Jensen's inequality.

The famous matrix $A_2$ conjecture stated that the vector Hilbert transform had a norm estimate in a matrix weighted $L^2(W)$ space bounded by the linear power of the matrix $A_2$ characteristic of the weight. This had been shown to be the case for scalar weights by the third author in \cite{Pe} using the classical scalar $A_2$ characteristic: 
$$\|\mathcal{H}\|_{L^2(w)\to L^2(w)}\lesssim [w]_{{A}_2}.$$
In the scalar case, the estimate itself required novel ideas at the time, but examples that showed that the growth had to be at least linear were known long before the estimate had been established \cite{Bu}. Indeed, power weights with exponent  approaching 1 from below provides such an example sequence, while taking as a test function its inverse. The sharpness of the exponent 1 is then a simple direct calculation for many operators, including the Hilbert transform: one can find a $A_2$ weight $w_s$ and a nonzero function $f_s\in L^2(w_s)$ such that
$$\|\mathcal{H}f_s\|_{L^2(w_s)}\sim [w_s]_{A_2}\|f_s\|_{L^2(w_s)}.$$

In \cite{NaPeTrVo} the concept of convex body sparse domination was established and resulted in the best to date upper estimate in the matrix case with a power of 3/2:
$$\|\mathcal{H}\|_{L^2(W)\to L^2(W)}\lesssim [W]_{\mathbf{A}_2}^{3/2}.$$
The authors had a number of proofs, but the obstacle of the non-commutativity could not be fully overcome in any of them. Finally the first, third, fourth and fifth authors constructed in \cite{DoPeTrVo} a complicated counterexample sequence that showed that 3/2 could not be improved, showing the matrix $A_2$ conjecture was false: one can find a matrix $A_2$ weight $W_s$ and a nonzero function $f_s\in L^2(W_s)$ such that
\begin{equation}
    \label{eq:sharp_matrix_A2}
    \|\mathcal{H}f_s\|_{L^2(W_s)}\sim [W_s]_{\mathbf{A}_2}^{3/2}\|f_s\|_{L^2(W_s)}. 
\end{equation}
This sequence of examples have an increasing number of singularities, carefully exploiting non-commutativity in a top down approach. 

``Fattened'' $A_2$ characteristics, where one no longer uses classical averages of the weight to compute the $A_2$ characteristic but rather bump averages became a theme of investigation for several reasons. An example of such a characteristic is the Poisson $A_2$ characteristic defined by
\begin{equation}
    \label{eq:scalar_Poisson_characteristic}
    [w]_{{A}_2}^{\text{fat}}=\sup_{x\in\R,~t>0}w(x,t)w^{-1}(x,t),
\end{equation}
where 
$$w(x,t) := \frac{1}{\pi}\int_{\R}\frac{t}{(x-y)^2+t^2}w(y)\,\mathrm{d}y,\quad x\in\R,~t>0.$$
For instance D.~Sarason conjectured \cite[Section 7.9]{HaNi} that the Hilbert transform mapped boundedly in $L^2$ in a two weight setting if and only if the weights had a joint Poisson $A_2$ characteristic. The necessity was shown by the fourth author to be true, the proof is presented with attribution in \cite[Section 7.9]{HaNi}. However, the sufficiency turned out to be false, as established by Nazarov in the unpublished manuscript \cite{Na}. Later in a series of works \cite{LaI, LaII, LaIII, Hy} the two weight problem for the Hilbert transform was solved, naming the joint Poisson $A_2$ condition as only one of several necessary and sufficient conditions for boundedness.

In a slightly different direction, Cruz-Uribe and P\'{e}rez \cite{CrPe} conjectured that one would obtain a sufficient condition for the two weight boundedness of the Hilbert transform and general Calder\'{o}n--Zygmund operators if in~\eqref{eq:scalar_A2_characteristic} one replaced the averages with respect to the $L^1$ norm by averages with respect to ``bigger'' Orlicz space norms. This was inspired by an earlier result of Neugebauer \cite{Ne}, who had shown that upon replacing the $L^1$ averages in~\eqref{eq:scalar_A2_characteristic} by $L^r$ averages for some $1<r<\infty$, one does obtain a sufficient condition. The aforementioned conjecture of Cruz-Uribe and P\'{e}rez \cite{CrPe} was confirmed almost simultaneously by Nazarov, Reznikov, the fourth and the fifth author \cite{NaReTrVo} and Lerner \cite{Le}. Inspired by these results, Cruz-Uribe, Reznikov and the fifth author \cite{CrReVo} asked whether it would be possible to ``separate'' the two Orlicz space averages, that is assuming the simultaneous finiteness of two characteristics in each of which one weight is integrated in the $L^1$ sense and the other in the Orlicz space sense, and still get a sufficient condition. This became known as the ``separated bump conjecture''. Since then several partial results have been achieved regarding this conjecture, see for instance \cite{AnCrMo, CrReVo, LaIV, RaSp, TrVo1}. In particular, the fourth and the fifth author showed in \cite{TrVo1} that this conjecture is true for sparse square or more general $p$-functions. This raised the hope that a proof of the separated bump conjecture could be achieved by dominating Calder\'{o}n--Zygmund through sparse $p$-functions (instead of sparse linear operators). However, this domination was disproved for the Hilbert transform by the second author in \cite{Ka}. Nevertheless, to the best of our knowledge the separated bump conjecture remains open as of the time of writing.

Going back to the case of one weight, it is remarkable that proving the norm estimate with exponent 1 for the Hilbert transform with scalar weight was a much simpler task if one allowed the Poisson characteristic \eqref{eq:scalar_Poisson_characteristic}, as shown by the third author and J. Wittwer \cite{PeWit}. In fact, this estimate predates the one in \cite{Pe}.
This is due to the connection between the Hilbert transform and the Poisson equation via Brownian motion or a Littlewood Paley formula, as well as the conformal invariance of the Poisson $A_2$ characteristic. However, the classical and the Poisson characteristics are not comparable with each other linearly due to the slow decay of the Poisson kernel. There is such a linear relation if one ``bumps'', for instance, the heat kernel \cite{PeVo}, but the Hilbert transform fits with the Poisson kernel, not the heat kernel. Even for simple power weights, the square of the classical characteristic is needed to bound the Poisson characteristic. (For a detailed computation we refer to the appendix of the PhD thesis of the second author \cite{Ka_PhD}.) It was a deep and complex task by the second and the fourth author \cite{KaTr} to establish that even in the presence of the Poisson $A_2$ characteristic, the linear estimate for the Hilbert transform is still optimal: one can find a Poisson $A_2$ weight $w_s$ and a nonzero function $f\in L^2(w_s)$ such that
$$\|Hf_s\|_{L^2(w_s)}\sim [w_s]_{{A}_2}^{\text{fat}}\|f_s\|_{L^2(w_s)}.$$
In this paper we show that the Poisson matrix $A_2$ characteristic still requires a 3/2 exponent in the norm estimate of the Hilbert transform. The Poisson matrix characteristic is defined by
$$ [W]_{\mathbf{A}_2}^{\text{fat}}=\sup_{x\in\R,~t>0}\|W(x,t)^{1/2}W^{-1}(x,t)^{1/2}\|^2$$ and it is shown that one can find a matrix Poisson $A_2$ weight $w_s$ and a nonzero function $f\in L^2(W_s)$ such that
$$\|Hf_s\|_{L^2(W_s)}\sim [W_s]_{\mathbf{A}_2^{\text{fat}}}^{3/2}\|f_s\|_{L^2(W_s)}.$$
Actually, unlike \cite{Na} but similarly to \cite{KaTr} we give an explicit construction of such an example. Inspired by the strategy of \cite{KaTr}, we proceed in several steps:
\begin{itemize}
    \item The starting point is an example featuring~\eqref{eq:sharp_matrix_A2} with the Hilbert transform replaced by a \emph{dyadic model} and the matrix $A_2$ characteristic replaced by the \emph{dyadic} $A_2$ characteristic. Such an example we borrow from \cite{DoPeTrVo}. We call it the ``large step'' example.

    \item Next, we want to obtain a dyadic example as in the first point, but this time featuring \emph{dyadically smooth} weights. This is inspired by \cite{KaTr}, which was itself inspired by \cite{Na}. However, since this time we are dealing with \emph{matrix} valued weights, it is not entirely clear what the condition of dyadic smoothness should look like. All inequalities that one writes involving matrices require great care.

    Even if one agrees on what dyadic smoothness in the matrix setting should be, it is not immediately clear that the ``small step'' transform from \cite{KaTr} can be reasonably applied to the present setting. The ``large step'' example that was used in \cite{KaTr}, essentially one of the classical examples due to \cite{Bu}, had on the level of the martingale representation of the involved functions a very simple structure. However, the ``large step'' example we use here from \cite{DoPeTrVo} is much more complicated. As a result, our analog of the ``small step'' transform is considerably more involved than in \cite{KaTr}. Our methods lie in the intersection of convex analysis, algebraic topology and probability theory. We consider general random walks on barycentric subdivisions of simplices and use in a crucial way the properties of martingales as ``fair games''. The proof that this transform indeed produces dyadically smooth matrix weights requires in turn some delicate functional analytic manipulations.

    \item Finally, we use the \emph{iterated remodeling} technique from \cite{DoPeTrVo}. This was itself inspired from \cite{KaTr}, where a refinement of the method of remodeling from \cite{Na} was developed. The remodeling method ultimately goes back to \cite{Bo}. Finally, we show that the application of iterated remodeling on the example of the previous point upgrades dyadic smoothness to ``full'' smoothness and consequently the dyadic matrix $A_2$ characteristic not only just to the classical matrix $A_2$ characteristic as in \cite{DoPeTrVo} but in fact even to the ``fattened'' matrix $A_2$ characteristic we are considering here.
\end{itemize}

In the next section we give the most important definitions as well as a rigorous statement of our main result.

\section{Definitions and main result}
\label{s:intro}

Recall that a ($d$-dimensional) matrix weight on $\R$  
is a locally integrable function on $\R$ 
with values in the set of positive definite $d\times d$ matrices.

The weighted space $L^2(W)$ 
is defined as the space of all measurable functions $f:\R \to \mathbb{F}^d$, (here $\mathbb{F}=\R$, or $\mathbb{F}=\C$) 
for which
\[
\| f\|_{L^2(W)}^2:=\int (W(x)f(x), f(x))_{\mathbb{F}^d} \,\mathrm{d} x <\infty\,;
\]
here $(\cdot, \cdot)_{\mathbb{F}^d}$ means the standard inner, respectively hermitian product in $\mathbb{F}^d$. 

A matrix weight $W$ is said to satisfy the \emph{matrix $\mathbf{A}_2$ condition} (write $W\in\mathbf{A}_2$) if 
\begin{align}\label{e:A2}
[W]_{\mathbf{A}_2} := \sup_I \left\| \langle W\rangle_ I^{1/2} \langle W^{-1}\rangle_ I^{1/2} \right\|^2  < 
\infty\, ,
\end{align}
where $I$ ranges over all intervals.
The quantity $[W]_{\mathbf{A}_2}$ is called the \emph{$\mathbf{A}_2$ characteristic of the weight} 
$W$.  In the scalar case, when $W$ is a scalar weight $w$, this coincides with the classical $A_2$ 
characteristic $[w]_{A_2}$.

The matrix weight $W$ is said to satisfy the \emph{``fattened'' matrix $\mathbf{A}_2$ condition} (write $W\in\mathbf{A}^{\text{fat}}_2$) if 

\[[W]_{\mathbf{A}_2}^{\text{fat}}=\sup_{x\in\R,~t>0}\|W(x,t)^{1/2}W^{-1}(x,t)^{1/2}\|^2\]
is finite, where abusing notation we denote the Poisson extension of $W$ on the upper half-plane by the same letter, that is
\begin{equation*}
    W(x,t) := \frac{1}{\pi}\int_{\R}\frac{t}{(x-y)^2+t^2}W(y)\,\mathrm{d}y,\quad x\in\R,~t>0.
\end{equation*}
Let us observe here that is not important that one uses this particular integration kernel. In fact, our methods apply equally well for any ``reasonable'' approximation of the identity.

Let $\mathcal{H}$ denote the Hilbert transform, 
\begin{align*}
\mathcal{H} f(s) =\frac1\pi \text{ p.v.}\int_\R \frac{f(t)}{s-t} \,\mathrm{d} t,\quad s\in\R.
\end{align*}
In this paper, we show that the exponent of $3/2$ persists if we ``fatten'' the matrix $A_2$ characteristic.  More precisely, our main result is:

\begin{thm}\label{MainTheorem1}
There exists a constant $c>0$ such that for all sufficiently large 
$Q>0$ there exist a $2\times 2$ matrix weight $W=W_{Q}$, 
$[W]_{\mathbf{A}_2}^{\emph{fat}}\leq Q$ and 
a function $f\in L^2(W)$, $f:\R\to \R^2$, $f\ne 0$ such that
\begin{align*}
\| \mathcal{H} f \|_{L^2(W)} \geq c\,Q^{3/2} \|f\|_{L^2(W)}. 
\end{align*}
\end{thm}

In fact, by picking a sufficiently small $c$ we can  state it for all $Q\ge 1$. By a simple 
reduction, we can state it for all dimensions $d\ge 2 $ of matrices.

\section{Preliminaries}

\subsection{``Fattened'' averages}
	
	Given a $(d\times d)$ matrix valued function $W$ on $\R$ with $W(x)>0$ for a.~e.~$x\in\R$, we define
	\begin{equation*}
		\langle W \rangle^{\text{fat}}_{\lambda}:=\frac{1}{\pi}\int_{\R}\frac{\mathrm{Im}(\lambda)}{|x-\lambda|^2}W(x)\,\mathrm{d}x,\quad\lambda\in\C_{+},
	\end{equation*}
    where $\C_{+} := \{\lambda\in\C:~\mathrm{Im}(\lambda)>0\}$. A simple calculation shows that
	\begin{equation}
		\label{fattened averages dominate usual averages}
		\langle W \rangle_{I_{\lambda}}
		\lesssim\langle W \rangle^{\text{fat}}_{\lambda},
	\end{equation}
	where $I_{\lambda}:=[\text{Re}(\lambda)-\text{Im}(\lambda),\text{Re}(\lambda)+\text{Im}(\lambda)]$, for all $\lambda\in\C_{+}$. The fattened matrix $A_2$ characteristic can be then rewritten as
	\begin{equation*}
		[W]^{\text{fat}}_{\mathbf{A}_2}=\sup_{\lambda\in\C_{+}}\Vert(\langle W \rangle^{\text{fat}}_{\lambda})^{1/2}(\langle W^{-1} \rangle^{\text{fat}}_{\lambda})^{1/2}\Vert^2<\infty.
	\end{equation*}
	Using \eqref{fattened averages dominate usual averages} (first for $W$ and then for $W^{-1}$) coupled with \cite[Lemma 3.1]{DoPeTrVo} we deduce
	\begin{align*}
		\langle W \rangle_{I_{\lambda}}
		\lesssim\langle W \rangle^{\text{fat}}_{\lambda}\leq[W]_{A_2}^{\mathrm{fat}}(\langle W^{-1}\rangle^{\text{fat}}_{\lambda})^{-1}\lesssim[W]_{A_2}^{\mathrm{fat}}(\langle W^{-1}\rangle_{I_{\lambda}})^{-1},
	\end{align*}
	for all $\lambda\in\C_{+}$. One more application of \cite[Lemma 3.1]{DoPeTrVo} yields then
	\begin{align*}
		[W]_{\mathbf{A}_2}\lesssim[W]^{\text{fat}}_{\mathbf{A}_2}.
	\end{align*}

	\subsection{The doubling condition}
	
	A $(d\times d)$ matrix weight $W$ on $\R$ is said to satisfy \emph{the doubling condition with constant $C>0$} if
	\begin{equation*}
		W(2I)\leq C\,W(I),
	\end{equation*}
	for all intervals $I\subseteq\R$. Here we denote $W(I):=\int_{I}W(x)\,\mathrm{d}x$. The best such constant $C$ is denoted by $D_{W}$.
	
	A standard calculation shows that if $D_{W}<4$, then
	\begin{equation}
		\label{usual averages dominate fattened averages under doubling}
		\langle W \rangle^{\text{fat}}_{\lambda}\lesssim\langle W \rangle_{I_{\lambda}},\quad\forall\lambda\in\C_{+}.
	\end{equation}
	For the sake of completeness, we briefly review this calculation: writing $\lambda=t+iy$ and $I=I_{\lambda}$, we compute
	\begin{align*}
		\pi\,\langle W \rangle^{\text{fat}}_{\lambda}&=\int_{\R}\frac{y}{(t-x)^2+y^2}W(x)\,\mathrm{d}x\\
		&=
		\int_{I}\frac{y}{(x-t)^2+y^2}W(x)\,\mathrm{d}x+\sum_{n=1}^{\infty}\int_{2^nI\setminus 2^{n-1}I}\frac{y}{(x-t)^2+y^2}W(x)\,\mathrm{d}x\\
		&\leq\frac{1}{y}\int_{I}W(x)\,\mathrm{d}x+\sum_{n=1}^{\infty}\int_{2^nI}\frac{y}{(2^{n-1}y)^2}W(x)\,\mathrm{d}x\\
		&=\frac{1}{y} W(I)+\frac{4}{y}\sum_{n=1}^{\infty}\frac{1}{4^{n}}W(2^{n}I)
		\leq\frac{1}{y} W(I)+\frac{4}{y}\sum_{n=1}^{\infty}\frac{1}{4^{n}}(D_{W})^nW(I)\\
		&=\left(2+8\sum_{n=1}^{\infty}\left(\frac{D_{W}}{4}\right)^n\right)\langle W\rangle_{I},
	\end{align*}
	with $2+8\sum_{n=1}^{\infty}\left(\frac{D_{W}}{4}\right)^n<\infty$.
	
	Therefore, if $W$ is a matrix $\mathbf{A}_2$ weight with $D_{W}<4$ and $D_{W^{-1}}<4$, then an application of \cite[Lemma 3.1]{DoPeTrVo} yields
	\begin{equation*}
		[W]^{\text{fat}}_{\mathbf{A}_2}\lesssim[W]_{\mathbf{A}_2}.
	\end{equation*}
	Thus, to achieve the passage from the classical matrix $A_2$ condition to the ``fattened'' one, it suffices to achieve control of doubling constants.

	\subsection{Smoothness constants}
	
	As in the scalar case \cite{KaTr, Na}, we control so-called \emph{smoothness constants}, instead of directly controlling doubling constants. These smoothness constants were originally defined for scalar weights by Nazarov \cite{Na}.
	
	Given a matrix weight $W$, we define the \emph{smoothness constant} $S_{W}$ as the best constant $C>0$ such that 
	\begin{equation*}
		\langle W\rangle_{I_{+}}\leq C\langle W\rangle_{I_{-}}\quad\text{and}\quad\langle W\rangle_{I_{-}}\leq C\langle W\rangle_{I_{+}},
	\end{equation*}
	for all intervals $I\subseteq\R$. It is easy to see that $D_{W}\leq S_{W}+1$: setting $I_1:=(2I)_{--}$ and $I_2:=(2I)_{++}$ and observing that $I_{-}=(2I)_{-+}$ and $I_{+}=(2I)_{+-}$, we have
	\begin{align*}
		W(2I)&=W(I_1)+W(I_{-})+W(I_{+})+W(I_2)\\
		&\leq S_{W}W(I_{-})+W(I_{-})+W(I_{+})+S_{W}W(I_{+})=(1+S_{W}) W(I).
	\end{align*}

	We define the \emph{dyadic smoothness constant} $S_{W}^{\text{dy}}$ as the best constant $C$ such that
	\begin{equation*}
		\langle W\rangle_{I_{+}}\leq C\langle W\rangle_{I_{-}}\quad\text{and}\quad\langle W\rangle_{I_{-}}\leq C\langle W\rangle_{I_{+}},
	\end{equation*}
	for all $I\in\cD$. It is immediate that $1\leq S_{W}^{\text{dy}}\leq S_{W}$.
	
	Finally, we define the \emph{strong dyadic smoothness constant} $S_{W}^{\text{sdy}}$ as the best constant $C$ such that
	\begin{equation*}
		\langle W\rangle_{I}\leq C\langle W\rangle_{J}\quad\text{and}\quad\langle W\rangle_{J}\leq C\langle W\rangle_{I},
	\end{equation*}
	for all \emph{adjacent} intervals $I,J\in\cD$ (this means that $I,J$ are disjoint and share an endpoint) with $|I|=|J|$.

    It was a crucial observation of Nazarov in the scalar case \cite{Na} that sufficient control over the strong dyadic smoothness constant ensures control over the smoothness constant. This remains true in the matrix case:

    \begin{lem}
		Let $\e>0$. Then, there exists $\delta=\delta(\e)>0$, such that the following holds. If $W$ is a matrix weight with $S_{W}^{\emph{sdy}}<1+\delta$, then $S_{W}<1+\e$.
	\end{lem}

	\begin{proof}
		The proof is almost identical to the one in the scalar case \cite{Na}. Nevertheless, we include most details.
		
		Pick $\delta\in\left(0,\frac{1}{4}\right)$ with
		\begin{equation*}
			(1-2\sqrt{\delta})(1+\delta)^{-2/\sqrt{\delta}}>(1+\e)^{-1/2}
		\end{equation*}
		and
		\begin{equation*}
			(1+2\sqrt{\delta})(1+\delta)^{2+2/\sqrt{\delta}}<(1+\e)^{1/2}.
		\end{equation*}
		
		\medskip
		
		\textbf{Claim.} For any interval $I\subseteq\R$ and for any $J\in\cD$ containing one of the two endpoints of $I$ with $|J|\leq\sqrt{\delta}|I|\leq 2|J|$, we have
		\begin{equation*}
			\langle W\rangle_{J}\leq(1+\e)^{1/2}\langle W\rangle_{I}
		\end{equation*}
		and
		\begin{equation*}
			\langle W\rangle_{I}\leq(1+\e)^{1/2}\langle W\rangle_{J}.
		\end{equation*}
		
		\medskip
		
		Let us assume for a moment the claim. Let $I\subseteq\R$ be an interval. Pick $J\in\cD$ containing the center of $I$ with $2|J|\leq\sqrt{\delta}|I|\leq 4|J|$. The claim, applied to $I_{+},J$, respectively $I_{-},J$, yields
		\begin{equation*}
			\langle W\rangle_{J}\leq(1+\e)^{1/2}\langle W\rangle_{I_{-}}\quad\text{and}\quad\langle W\rangle_{I_{-}}\leq(1+\e)^{1/2}\langle W\rangle_{J},
		\end{equation*}
		respectively
		\begin{equation*}
			\langle W\rangle_{J}\leq(1+\e)^{1/2}\langle W\rangle_{I_{+}}\quad\text{and}\quad\langle W\rangle_{I_{+}}\leq(1+\e)^{1/2}\langle W\rangle_{J}.
		\end{equation*}
		It is then immediate that $\langle W\rangle_{I_{-}}\leq(1+\e)\langle W\rangle_{I_{+}}$ and $\langle W\rangle_{I_{+}}\leq(1+\e)\langle W\rangle_{I_{-}}$.
		
		We now show the claim. Set
		\begin{equation*}
			J_{*}:=\{K\in\cD:~|K|=|J|,\,K\subseteq I\},\quad I_{*}:=\bigcup_{K\in J_{*}}K.
		\end{equation*}
		Then $J_{*}\neq\emptyset$, since $|J|<\frac{1}{2}|I|$. Clearly
		\begin{equation*}
			\# J_{*}\leq\frac{|I|}{|J|}\leq\frac{2}{\sqrt{\delta}}.
		\end{equation*}
		For all $K\in J_{*}$, we can find $\ell\in\{1,\ldots,\#J_{*}\}$ and
		\begin{equation*}
			J_1,\ldots,J_{\ell+1}\in\cD\text{ of length }|J|
		\end{equation*}
		such that
		\begin{equation*}
			J_1=K,\quad J_{\ell+1}=J
		\end{equation*}
		and $J_i,J_{i+1}$ are adjacent, for all $i=1,\ldots,\ell$, thus
		\begin{equation*}
			\langle W\rangle_{J_i}\geq(1+\delta)^{-1}\langle W\rangle_{J_{i+1}},\quad i=1,\ldots,\ell,
		\end{equation*}
		thus
		\begin{equation*}
			\langle W\rangle_{K}\geq(1+\delta)^{-\ell}\langle W\rangle_{J}\geq(1+\delta)^{-2\sqrt{\delta}}\langle W\rangle_{J}.
		\end{equation*}
		Thus
		\begin{align*}
			\langle W\rangle_{I_{*}}&=\frac{1}{|I_{*}|}\sum_{K\in J_{*}}W(K)\geq\frac{(1+\delta)^{-2/\sqrt{\delta}}}{|I_{*}|}\sum_{K\in J_{*}}|K|\langle W\rangle_{J}\\
			&=\frac{(1+\delta)^{-2/\sqrt{\delta}}}{|I_{*}|}|J|(\# J_{*})\langle W\rangle_{J}
			=(1+\delta)^{-2/\sqrt{\delta}}\langle W\rangle_{J}.
		\end{align*}
		Also $|I_{*}|\geq|I|-2|J|\geq(1-2\sqrt{\delta})|I|$ and $I_{*}\subseteq I$, so
		\begin{align*}
			\langle W\rangle_{I}\geq\frac{|I_{*}|}{|I|}\langle W\rangle_{I_{*}}\geq
			(1-\sqrt{\delta})(1+\delta)^{-2/\sqrt{\delta}}\langle W\rangle_{J}\geq(1+\e)^{-1/2}\langle W\rangle_{J},
		\end{align*}
		therefore $\langle W\rangle_{J}\leq(1+\e)^{1/2}\langle W\rangle_{I}$.
		
		Finally, adapting the corresponding argument from the scalar case as above, we also get $\langle W\rangle_{I}\leq(1+\e)^{1/2}\langle W\rangle_{J}$.
	\end{proof}
	
\section{``Large step'' examples}
\label{sec:large_step_examples}

In this section we recall the construction of the example yielding the lower bound for the dyadic Hilbert transform, $\mathcal{H}^{\mathrm{dy}}$, from \cite{DoPeTrVo}.  Our notation and terminology follow largely \cite{DoPeTrVo}, where we refer the reader for details and proofs. We place particular emphasis on the ``geometric structure'' of the constructed dyadic martingales.

\subsection{Construction of the weights}

We begin by recalling the construction of the $2\times 2$ matrix weights $W$ and $V=W^{-1}$ from \cite[Section 4]{DoPeTrVo}. Let $Q$ be a (sufficiently) large positive real number. Let $\cD$ be the family of all dyadic subintervals of the unit interval $I^{0}:=[0,1)$. Given any dyadic subinterval $I$ of $I^{0}$, such that the averages $\langle W\rangle_{I}$ and $\langle V\rangle_{I}$ are already known, the assignment of averages to the dyadic children $I_{\pm}$ of $I$ follows exactly one of the following principles:
\begin{itemize}
    \item the \emph{rotation} operation, in which case we call $I$ a \emph{blue interval} (or in the terminology of \cite{DoPeTrVo}, a \emph{stopping interval})
    \item the \emph{streching} operation, in which case we call $I$ a \emph{red interval}
    \item the \emph{terminating} operation \cite[Lemma 3.4]{DoPeTrVo}, in which case we call $I$ a \emph{green interval} (or in the terminology of \cite{DoPeTrVo}, a \emph{terminal interval})
    \item $W,V$ are constant on $I$, in which case we call $I$ a \emph{black interval}.
\end{itemize}

The construction begins by declaring $I^{0}$ to be a blue interval. If $I$ is now any blue interval, then each of the two dyadic children $J=I_{\pm}$ of $I$ is declared to be a red interval.

If $J$ is any red interval, then $J_{+}$ is declared to be a blue interval, while $K=J_{-}$ is declared to be a green interval.

Finally, if $K$ is any green interval, then both children of $K$ (and with them all other dyadic descendants of $K$) are declared to be black intervals.

We set $\mathscr{S}_0:=\{I^0\}$. We denote the collection of all intervals $J=I_{\pm}$ with $I\in\mathscr{S}_n$ by $\mathscr{F}_{n+1}$. Moreover, the collection of all intervals $J_{+}$ with $J\in\mathscr{F}_{n+1}$ is denoted by $\mathscr{S}_{n+1}$.

As presented, this construction could proceed for infinitely many steps. As commented in \cite{DoPeTrVo}, this would create no issues. However, the construction of \cite{DoPeTrVo} terminates after finitely many steps. That means, there is some (large enough) positive integer $N_0$, such that the intervals $I\in\mathscr{S}_{N_0}$ are exceptionally declared to be green intervals, and all dyadic descendants of $I$ are declared to be black intervals. The termination after finitely many steps will turn out to be important in Subsection~\ref{subs:small_step_dyadic_smoothness} below.

Thus, we obtain finite sequences
\begin{gather*}
    \mathscr{S}_0,\mathscr{S}_1,\ldots,\mathscr{S}_{N_0}\\
    \mathscr{F}_1,\mathscr{F}_2,\ldots,\mathscr{F}_{N_0}.
\end{gather*}
Let us denote $\mathscr{S}=\bigcup_{n\geq0}\mathscr{S}_n$, $\mathscr{S}_{>}:=\bigcup_{n>0}\mathscr{S}_n$, $\mathscr{F}:=\bigcup_{n>0}\mathscr{F}_{n}$. We also denote by $\mathscr{G}$ the family of all green intervals except for those in $\mathscr{S}_{N_0}$. Finally, we denote by $\mathscr{E}$ the family of all black intervals, including the children of the intervals in $\mathscr{S}_{N_0}$.

The above construction is adapted to several parameters that depend only on $Q$, as explained in \cite[Section 4]{DoPeTrVo}, so that $[W]^{\mathrm{dy}}_{\mathbf{A}_2}=Q$.

\subsection{The lower bound for the dyadic Hilbert transform}
\label{subs:damage_large_step}

Next, we recall the dyadic Hilbert transform $\mathcal{H}^{\mathrm{dy}}$ from \cite[Section 6]{DoPeTrVo}. One has
\begin{equation*}
    \mathcal{H}^{\mathrm{dy}}f=c_1(Sf-S^{*}f)+c_2S_0f,
\end{equation*}
where the operators $S,S^{*}$ and $S_{0}$ are defined by
\begin{gather*}
    Sf:=\sum_{I\in\cD_{\mathrm{o}}}(f,h_{I})_{L^2}[h_{I_{+}}-h_{I_{-}}],\\
    S^{*}f:=\sum_{I\in\cD_{\mathrm{o}}}[(f,h_{I_{+}})_{L^2}-(f,h_{I_{-}})_{L^2}]h_{I},\\
    S_0f:=\sum_{I\in\cD_{\mathrm{o}}}[(f,h_{I_{+}})_{L^2}h_{I_{-}}-(f,h_{I_{-}})_{L^2}h_{I_{+}}],
\end{gather*}
and $c_1,c_2$ are real positive constants whose values are determined appropriately in \cite{DoPeTrVo}; here we will use the same ones. We emphasize that these operators are ```sliced'' and indexed by odd dyadic intervals $I\in \cD_{\mathrm{o}}$ if and only if $I\in \cD$ and $|I|=2^{-k}$ with $k$ an odd integer.

For some nonzero vector $e\in\R^2$ chosen as in \cite[Lemma 5.1]{DoPeTrVo}, we define as in \cite[Sections 6 and 7]{DoPeTrVo} the functions $\mathbf{f}:=\mathbf{1}_{I^{0}}W^{-1}e$ and $\mathbf{g}:=W\mathcal{H}^{\mathrm{dy}}\mathbf{f}$. Observe that by construction only intervals in $\mathscr{F}\cup\mathscr{E}$ are odd intervals. Moreover, the function $f$ is constant on any interval $I\in\mathscr{E}$. Thus, we actually have
\begin{gather*}
    S\mathbf{f}=\sum_{I\in\mathscr{F}}(\mathbf{f},h_{I})_{L^2}[h_{I_{+}}-h_{I_{-}}],\\
    S^{*}\mathbf{f}=\sum_{I\in\mathscr{F}}[(\mathbf{f},h_{I_{+}})_{L^2}-(\mathbf{f},h_{I_{-}})_{L^2}]h_{I},\\
    S_0f=\sum_{I\in\mathscr{F}}[(\mathbf{f},h_{I_{+}})_{L^2}h_{I_{-}}-(\mathbf{f},h_{I_{-}})_{L^2}h_{I_{+}}].
\end{gather*}

Delicate computations in \cite[Section 6]{DoPeTrVo} show that
\begin{equation*}
    |(\mathcal{H}^{\mathrm{dy}}\mathbf{f},\mathbf{g})_{L^2}|\gtrsim Q^{3/2}\|\mathbf{f}\|_{L^2(W)}\|\mathbf{g}\|_{L^2(W^{-1})}.
\end{equation*}

\subsection{The geometric picture of the full martingale}

Let us consider the bounded function $F:=(W,V,\mathbf{f},\mathbf{g})$, which is defined on $I^{0}$ and takes values in the vector space $\mathscr{X}:=\R^{2\times 2}\times\R^{2\times 2}\times\R^2\times\R^2$. The martingale induced by it with respect to the dyadic filtration on $I^{0}$ can be fully described through the set of averages $\{\langle F\rangle_{I}\}_{I\in\cD}$. Let us observe the following:
\begin{itemize}
\item By construction, each of the functions $W,V,\mathbf{f}$ is constant on each interval $I\in\mathscr{E}$.

\item For each $I\in\cD$ we have $(\mathcal{H}^{\mathrm{dy}}\mathbf{f},h_{I})_{L^2}=0$ whenever $I$ does not satisfy any of the following: $I\in\mathscr{S}$ or $I\in\mathscr{F}$ or $I\in\mathscr{G}$. In particular, the function $\mathcal{H}^{\mathrm{dy}}\mathbf{f}$ is constant on each interval $I\in\mathscr{E}$. Since this holds also for $W$, we conclude that this is true for $\mathbf{g}$ as well.
\end{itemize}
 Thus, the arrangement of the points $\{\langle F\rangle_{I}\}_{I\in\cD}$ in $\mathscr{X}$ consists of the following parts:
\begin{itemize}
    \item There is one straight line segment with middle point $\langle F\rangle_{I^{0}}$. Its endpoints are occupied by the averages of $F$ over the intervals $I^{0}_{+}$ and $I^{0}_{+}$, which both belong to the family $\mathscr{F}$. Note that this segment is by construction non-degenerate, i.e.~$\langle F\rangle_{I^{0}_{+}}\neq\langle F\rangle_{I^{0}_{-}}$.

    \item To each $K\in\mathscr{F}$, there corresponds a (possibly degenerate) tetrahedron, denoted in the sequel by $\mathscr{C}(K)$, with vertices $\langle F\rangle_{K_{\pm\pm}}$ and center (of mass) at $\langle F\rangle_{K}$. The averages $\langle F\rangle_{K_{\pm}}$ are middle points of two edges of this tetrahedron. We have $K_{--},K_{-+}\in\mathscr{E}$. Moreover, if $K\notin\mathscr{F}_{N_0}$, then $K_{++},K_{+-}\in\mathscr{F}$ (so the averages corresponding to these intervals are centers of further tetrahedra of the same type), while if $K\in\mathscr{F}_{N_0}$, then $K_{++},K_{+-}\in\mathscr{E}$.
\end{itemize}
We emphasize that the intervals in $\mathscr{E}$ form a partition of $I^{0}$ and
\begin{equation}
    \label{eq:F_decomposition}
    F=\sum_{L\in\mathscr{E}}\langle F\rangle_{L}\mathbf{1}_{L}.
\end{equation}

\section{Discrete ``small step'' random walks on a \texorpdfstring{$3$}{3}-simplex}
\label{sec:small_step_random_walks_tetrahedron}

As explained above, the building block of the ``large step'' example of the last section is essentially a random walk (more precisely, a \emph{martingale}) on a (possibly degenerate) $3$-simplex in some vector space, which begins on the center of the simplex and reaches in only one step (or two steps, depending on point of view) almost surely one of the four vertices, each with same probability $1/4$. Since the affine dimension of $\R^3$ is exactly 3, it is the correct ambient space to study random walks on $3$-simplices. In this section we describe a particular discrete random walk (more precisely, a discrete martingale) on any given nondegenerate $3$-simplex in $\R^3$, which also begins on the center and ends almost surely on one of the vertices, each with the same probability $1/4$, but in many more steps. We employ the terminology ``small step'' for this random walk, since its increments in its main part are rescaled versions of the increments of the ``large step'' random walk by a factor that is much smaller than 1.

Although we could describe such ``small step'' random walks quite generally, we restrict ourselves to a concrete model suiting the ``small step'' example in the next section. In particular, the various choices for the increments of the random walk stipulated here are motivated by a desire to be able to estimate the action of the dyadic Hilbert transform on our ``small step'' functions as easily as possible in a later section. We begin the preparation for these estimates already in this section.

\subsection{Setup and barycentric coordinates}

We consider four \emph{affinely independent} points $a_1,a_2,a_3,a_4$ points in $\R^3$. Affine independence means that for all $\lambda_1,\lambda_2,\lambda_3,\lambda_4\in\R$ we have
\begin{equation*}
    \sum_{i=1}^{4}\lambda_i=0~\text{ and }~\sum_{i=1}^{4}\lambda_ia_i=0\quad\Longrightarrow\quad\lambda_i=0,~i=1,2,3,4.
\end{equation*}
Then, we can consider the nondegenerate $3$-simplex $\mathscr{K}:=\mathrm{Conv}(a_1,a_2,a_3,a_4)$ in $\R^3$ with \emph{vertices} $a_1,a_2,a_3,a_4$. Nondegeneracy means that $\mathscr{K}$ has nonempty topological interior as a subset of $\R^3$. Its \emph{barycenter is}
\begin{equation*}
    a_0:=\frac{1}{4}a_1+\frac{1}{4}a_2+\frac{1}{4}a_3+\frac{1}{4}a_4.
\end{equation*}
Due to affine independence of $a_1,a_2,a_3,a_4$ we can identify the simplex $\mathscr{K}$ with the set of its vertices. Thus, we will also be calling $\{a_1,a_2,a_3,a_4\}$ a $3$-simplex.

We observe that for all $a\in \mathscr{K}$ there exist unique $\lambda_1,\lambda_2,\lambda_3,\lambda_4\in[0,1]$, called \emph{barycentric cooordinates} of $a$, such that
\begin{equation*}
    \sum_{i=1}^{4}\lambda_i=1 \quad\text{and}\quad a=\sum_{i=1}^{4}\lambda_ia_i.
\end{equation*}
Observe that $a$ lies on a $2$-subsimplex, or \emph{face}, of $\mathscr{K}$ if and only if there is $i\in\{1,2,3,4\}$ with $\lambda_i=0$. Similarly, $a$ lies on a $1$-subsimplex, or \emph{edge}, of $\mathscr{K}$ if and only if there are $i,j\in\{1,2,3,4\}$ with $i\neq j $ and $\lambda_i=\lambda_j=0$. Finally, $a$ lies on a $0$-subsimplex, or \emph{vertex}, of $\mathscr{K}$ if and only if there exists $i\in\{1,2,3,4\}$ with $\lambda_j=0$ for $j\neq i$ (equivalently, $\lambda_i=1$).

Thus, describing a stochastic process whose terms take values on $\mathscr{K}$ amounts to describing four real-valued stochastic processes that at each time can be barycentric coordinates, i.e.~they take nonnegative real values and sum up to exactly 1.

In the rest of this section we keep this setup without further mention.

\subsection{Barycentric subdivision}

To make the notion of ``small step'' precise, we will need the concept of \emph{barycentric subdivision}, which plays a prominent role in other mathematical fields like algebraic topology and numerical analysis.

Let $d$ be a positive integer. Then, the collection $\cA$ of all points $a\in \mathscr{K}$ whose barycentric coordinates have the form
\begin{equation}
    \label{eq:barycentric_subdivision}
    \left(\frac{m_1}{4d},\frac{m_2}{4d},\frac{m_3}{4d},\frac{m_4}{4d}\right)\text{ with }m_i\in\N_{0},~i=1,2,3,4\text{ and }\sum_{i=1}^{4}m_i=4d
\end{equation}
is termed a \emph{barycentric subdivision of the $3$-simplex $\{a_1,a_2,a_3,a_4\}$}. Observe that the vertices and the barycenter of $\mathscr{K}$ belong to $\cA$. The ``large step'' random walk takes place only on the barycenter and the vertices. The barycentric subdivision enriches the simplex, turning it into in a ``lattice'' of many more points, on which the ``small step'' random walk will take place.

\subsection{Reaching the vertices}

Before describing our random walk we will need two technical results. We will use them below to justify why it reaches some vertex almost surely in finite time, and why each vertex is reached with the same probability, namely $1/4$.

\subsubsection{Finiteness of hitting times} As we will see, the random walk begins in the interior of the simplex and will reach, or \emph{hit}, in the generic case first a face of the simplex, then an edge, and finally a vertex. Each time a subsimplex is reached, a different rule for building the random walk will apply. Thus, hitting times, or in other words \emph{stopping times} govern these transitions. These stopping times will have to be finite almost surely.

To build our random walk, we will be using the Rademacher functions, or equivalently Haar functions. On the level of barycentric coordinates, we get four real-valued stochastic processes that very much resemble the classical symmetric $\Z$-valued random walk that begins at 0. It is well-known in the classical case that almost surely every path reaches every integer (see for example \cite{Klenke2020}). Here we need a slightly more general version of this fact. For the reader's convenience we give a statement and a proof below that are more than sufficient for our purposes, adapting one of the proofs in the classical case. We note that the crucial property lurking in the background is that such random walks are martingales that are very quickly $L^2$ divergent. For such martingales it is known that the probability of their remaining inside bounded regions decays rapidly. See for example \cite{Makarov1990} as well as \cite[Lemma 9]{KS2024} for different versions of such results.

\begin{lem}
\label{lem:finiteness_hitting_times}
Let $(\Omega,\cF,\mathbb{P})$ be a probability space. Let $\{\omega_n\}^{\infty}_{n=1}$ be a sequence of $\Z$-valued, independent, identically distributed, not $\mathbb{P}$-a.s.~constant random variables on $\Omega$ with $\mathbb{E}[\omega_1]=0$. Let $a\in\Z$. Consider the sequence $\{S_n\}_{n=0}^{\infty}$ of random variables on $\Omega$ given by
\begin{equation*}
	S_0\equiv a,\quad S_n:=a+\sum_{k=1}^{n}\omega_k,~n=1,2,\ldots.
\end{equation*}

\begin{enumerate}
    \item Assume that there exists $M\in\Z$ with $\omega_n\leq M$ $\mathbb{P}$-a.s., for every $n=1,2,\ldots$. Let $b\in\Z$ with $a<b$. Set
    \begin{equation*}
	    \tau:=\inf\{n\geq1:~S_n\geq b\}.
    \end{equation*}
    Then, we have $\tau<\infty$ $\mathbb{P}$-a.s. Moreover, if $\omega_n\leq 1$ $\mathbb{P}$-a.s., for every $n=1,2,\ldots$, then $S_{\tau}=b$ $\mathbb{P}$-a.s.

    \item Assume that there exists $m\in\Z$ with $\omega_n\geq m$ $\mathbb{P}$-a.s., for every $n=1,2,\ldots$. Let $c\in\Z$ with $a>c$. Set
    \begin{equation*}
	    \sigma:=\inf\{n\geq1:~S_n\leq c\}.
    \end{equation*}
    Then, we have $\sigma<\infty$ $\mathbb{P}$-a.s. Moreover, if $\omega_n\geq 1$ $\mathbb{P}$-a.s., for every $n=1,2,\ldots$, then $S_{\sigma}=c$ $\mathbb{P}$-a.s.
\end{enumerate}
\end{lem}

\begin{proof}
\begin{enumerate}
    \item First of all, let $\cF_0:=\{\emptyset,\Omega\}$, $\cF_n:=\sigma(\omega_1,\ldots,\omega_n)$, $n=1,2,\ldots$ and consider the filtration $\mathbb{F}:=\{\mathcal{F}_n\}^{\infty}_{n=1}$ on $\Omega$. Fix a positive real number $\theta$ and set
	\begin{equation*}
		f(\theta):=\mathbb{E}[e^{\theta\omega_1}].
	\end{equation*}
	Then, it is easy to verify that the stochastic process $\{M_n\}^{\infty}_{n=0}$ given by
	\begin{equation*}
		M_n:=\frac{e^{\theta S_n}}{(f(\theta))^n}\cdot e^{-\theta a},\quad n=0,1,2\ldots
	\end{equation*}
	is a $\mathbb{F}$-adapted martingale. Obviously, $\tau$ is a $\mathbb{F}$-stopping time. Therefore, from the optional stopping theorem we deduce that the stopped process $\{M_{\tau\wedge n}\}^{\infty}_{n=0}$ is also a $\mathbb{F}$-adapted martingale. Let us note that for $\mathbb{P}$-almost every $x\in\Omega$ with $\tau(x)<\infty$, we have $\tau(x)>0$ and $S_{\tau(x)-1}(x)\leq b$, therefore $S_{\tau(x)}(x)\leq b+M$. Thus, we deduce
	\begin{equation*}
		0<M_{\tau\wedge n}\leq\frac{e^{\theta\max\{b,b+M\}}}{(f(\theta))^n}\cdot e^{-\theta a}\quad\mathbb{P}\text{-a.s.,}
	\end{equation*}
    for all $n=0,1,2,\ldots$. In particular, $\{M_{\tau\wedge n}\}^{\infty}_{n=0}$ is uniformly integrable. Therefore, there is a random variable $M$, such that $M_{\tau\wedge n}\rightarrow M$ pointwise $\mathbb{P}$-a.s.~and in $L^1(\Omega)$ as $n\rightarrow\infty$. 
	
	Notice that by Jensen's inequality we have
	\begin{equation*}
		f(\theta)>e^{\mathbb{E}[\theta \omega_1]}=1,
	\end{equation*}
	the strict inequality being due to the fact that $\omega_1$ is not $\mathbb{P}$-a.s.~constant. Thus, for all $x\in\Omega$ with $\tau(x)=\infty$, since
	\begin{equation*}
		0\leq M_{\tau(x)\wedge n}(x)=\frac{e^{\theta S_{n}(x)}}{(f(\theta))^n}\cdot e^{-\theta a}\leq\frac{e^{\theta b}}{(f(\theta))^n}\cdot e^{-\theta a},\quad\forall n=0,1,2,\ldots,
	\end{equation*}
	we deduce
	\begin{equation*}
		\lim_{n\rightarrow\infty}M_{\tau(x)\wedge n}(x)=0.
	\end{equation*}
	It follows that $M(x)=0$ for $\mathbb{P}$-almost every $x\in\Omega$ with $\tau(x)=\infty$. Observe also that
	\begin{equation*}
		M(x)=\frac{e^{\theta S_{\tau(x)}}}{(f(\theta))^{\tau(x)}}\cdot e^{-\theta a}
	\end{equation*}
	(because $\tau(x)\wedge n=\tau(x)$, for all $n=\tau(x),\tau(x)+1,\ldots$), for $\mathbb{P}$-almost every $x\in\Omega$ with $\tau(x)<\infty$. Therefore, we obtain
	\begin{equation*}
		\mathbb{E}[M]=\mathbb{E}\left[\frac{e^{\theta S_{\tau}}}{(f(\theta))^{\tau}}\cdot e^{-\theta a}\mathbf{1}_{\{\tau<\infty\}}\right].
	\end{equation*}
	Since also $\mathbb{E}[M]=\mathbb{E}[M_0]=1$, we deduce
	\begin{equation*}
		\mathbb{E}\left[\frac{e^{\theta (S_{\tau}-a)}}{(f(\theta))^{\tau}}\mathbf{1}_{\{\tau<\infty\}}\right]=1.
	\end{equation*}
	Since $\theta>0$ was arbitrary, letting $\theta\rightarrow0^{+}$ and using the Dominated Convergence Theorem we obtain $\mathbb{E}[\mathbf{1}_{\{\tau<\infty\}}]=1$, that is $\tau<\infty$ $\mathbb{P}$-a.s.

    If $\omega_n\leq 1$ $\mathbb{P}$-a.s., for every $n=1,2,\ldots$, then for $\mathbb{P}$-almost every $x\in\Omega$ with $\tau(x)<\infty$ we have $S_{\tau(x)-1}(x)\leq b-1$, therefore
	\begin{equation*}
		b\leq S_{\tau(x)}(x)\leq b-1+1=b,
	\end{equation*}
	thus $S_{\tau(x)}(x)=b$.

    \item This follows immediately from the first part after observing that $-a<-c$ and
    \begin{equation*}
		\sigma=\inf\{n\geq 1:~-S_n\geq-c\}.
	\end{equation*}
\end{enumerate}
\end{proof}

\subsubsection{Probability of reaching each vertex}

Our random walk will begin on the center of the $3$-simplex and will reach one of the four vertices almost surely in finite time. Will one of the vertices receive unexpected preference and be reached with higher probability than the others? Remarkably, the following lemma shows that this will not be the case, provided that our stochastic process is a martingale.

\begin{lem}
\label{lem:probability_reach_each_vertex}
Let $(\Omega,\mathcal{F},\mathbb{F},\mathbb{P})$ be a filtered probability space. Let $\{X_n\}^{\infty}_{n=0}$ be a $\mathbb{R}^3$-valued, uniformly bounded, $\mathbb{F}$-adapted martingale. Let $a_0\in\R^3$. Assume that $X_0=a_0$ $\mathbb{P}$-a.s.~and that the limit function $X_{\infty}$ of $\{X_n\}^{\infty}_{n=0}$ satisfies
\begin{equation*}
	X_{\infty}\in\{a_1,a_2,a_3,a_4\}\quad\mathbb{P}\text{-a.s.},
\end{equation*}
where $a_1,a_2,a_3,a_4$ are affinely independent points in $\R^3$ with $\frac{1}{4}(a_1+a_2+a_3+a_4)=a_0$. Then, we have $\mathbb{P}(X_{\infty}=a_i)=\frac{1}{4}$, for all $i=1,2,3,4$.
\end{lem}

\begin{proof}
Set
\begin{equation*}
	p_i:=\mathbb{P}(X_{\infty}=a_i),\quad i=1,2,3,4.
\end{equation*}
By assumption, $p_1+p_2+p_3+p_4=1$. Moreover, by the martingale property we have
\begin{equation*}
	\mathbb{E}[X_{\infty}]=\mathbb{E}[X_0]=a_0,
\end{equation*}
that is
\begin{equation*}
	\sum_{i=1}^{4}p_ia_i=a_0=\sum_{i=1}^{4}\frac{1}{4}a_i.
\end{equation*}
Since $\displaystyle\sum_{i=1}^{4}p_ia_i$, $\displaystyle\sum_{i=1}^{4}\frac{1}{4}a_i$ are both convex combinations of the affinely independent points $a_1,a_2,a_3,a_4$, we deduce $p_i=\frac{1}{4}$, for all $i=1,2,3,4$.
\bigskip
\end{proof}

\subsection{A ``small step'' \texorpdfstring{$4$}{4}-adic random walk}

Now we come to describing our ``small step" random walk. For reasons of  habit in analysis, we use dyadic intervals to encode trajectories of sign tosses. We fix a dyadic interval $I_0$, such that $-\log_{2}\ell(I_0)$ is odd. The underlying probability space is $I_0$ with the Borel $\sigma$-algebra and normalized Lebesgue measure. We denote by $\cD_{\mathrm{o}}(I_0)$ all dyadic subintervals $I$ of $I_0$ such that $-\log_{2}\ell(I)$ is also odd. We recall the $L^{\infty}$-normalized Haar functions
\begin{equation*}
    \varepsilon_{I}:=\mathbf{1}_{I_+}-\mathbf{1}_{I_-}.
\end{equation*}
For each interval $I$, we denote by $\mathrm{ch}^k(I)$ the family of all dyadic subintervals $J$ of $I$ with $\ell(J)=2^{-k}\ell(I)$.

The filtration we will be considering will be $\mathbb{F}=\{\cF_n\}^{\infty}_{n=0}$ with $\cF_n$ being the $\sigma$-algebra generated by
\begin{equation*}
    \cD_{\mathrm{o},n}(I_0):=\{I\in\cD_{\mathrm{o}}(I_0):~\ell(I)=2^{-2n}\ell(I_0)\}.
\end{equation*}
We will define a $\mathbb{F}$-adapted martingale $\{\wt{a}_n\}_{n=0}^{\infty}$ whose terms take values in the barycentric subdivision $\cA$ defined in terms of barycentric coordinates in \eqref{eq:barycentric_subdivision}. For each $n=0,1,2,\ldots$, we will write
\begin{equation*}
    \wt{a}_n=\sum_{I\in\cD_{\mathrm{o},n}(I_0)}\wt{a}_{I},
\end{equation*}
where for each $I$, $\wt{a}_{I}$ is a $\cA$-valued function that is constant on $I$ and vanishes outside of it. Thus, in order to describe our random walk, we just need to describe the functions $\wt{a}_{I}$. For the sake of simplicity we will denote the constant value of $\wt{a}_{I}$ on $I$ also by $\wt{a}_{I}$. We will also be denoting by $\lambda_{i,I}$, $i=1,2,3,4$ the barycentric coordinates of $\wt{a}_{I}$.

The associated martingale differences $\wt{b}_n:=\wt{a}_{n}-\wt{a}_{n-1}$, $n=1,2,\ldots$ will be sums of the form
\begin{equation*}
    \wt{b}_n=\sum_{I\in\cD_{\mathrm{o},n-1}(I_0)}\wt{b}_{I},
\end{equation*}
where each $\wt{b}_{I}$ vanishes outside of $I$, has zero average on $I$ and is constant on each $J\in\mathrm{ch}^2(I)$. This will ensure that $\{\wt{a}_n\}^{\infty}_{n=0}$ is indeed a $\mathbb{F}$-martingale.

We describe our ``small step'' random walk on the barycentric subdivision $\cA$ inductively in the subsections below.

However, a few words are in order regarding the choice of the random generators. When jumping randomly from $I$ to one of the four grandchildren $J\in\mathrm{ch}^2(I)$, we naturally use one of the three random generators $\varepsilon_I$, $\varepsilon_{I_+}$ and $\varepsilon_{I_-}$. In particular, they allow us to move in three independent directions in the interior of the simplex (see next subsection). When we reach a face, we need to produce a two-dimensional random walk and it is tempting to use $\varepsilon_{I_+}$ and $\varepsilon_{I_-}$. However, by doing so, when applying the dyadic Hilbert transform $\mathcal{H}^{\mathrm{dy}}$ to the ``small step'' martingale, we would \emph{not} be able to estimate the corresponding terms. This is the reason why, after reaching a face, we will only use $\varepsilon_I$ as a random generator to produce a two-dimensional random walk. Similarly when reaching an edge, we only use $\varepsilon_I$. We refer to Section \ref{subs:one_block_damage} where those observations are put to good use.

\subsubsection{Inside the simplex}

The random walk begins with the constant function $\wt{a}_0\equiv a_0$, so $\wt{a}_{I_0}:=a_0$.

Assume now that for some $I\in\cD_{\mathrm{o}}(I_0)$ we have defined $\wt{a}_{I}\in\cA$ lying in the \emph{interior} of the $3$-simplex. Observe that
\begin{equation}
    \label{eq:barycentric_cooordinates_inside_simplex}
    \lambda_{i,I}\geq\frac{1}{4d},\quad i=1,2,3,4,\quad\sum_{i=1}^{4}\lambda_{i,I}=1.
\end{equation}
Observe that in particular
\begin{equation*}
    \lambda_{i,I} = 1 - \sum_{\substack{j=1\\j\neq i}}^{4}\lambda_{j,I}
    \leq 1- \frac{3}{4d},\quad\forall i=1,2,3,4.
\end{equation*}
We define
\begin{equation*}
    \wt{b}_{I}:=\varepsilon_{I}\frac{a_1+a_2-a_3-a_4}{4d}+\varepsilon_{I_{+}}\frac{a_2-a_1}{2d}+\varepsilon_{I_{-}}\frac{a_4-a_3}{2d}.
\end{equation*}
Let us check that this definition makes sense. One can write
\begin{equation*}
    \wt{b}_{I}=\left(\frac{\varepsilon_{I}}{4d}-\frac{\varepsilon_{I_{+}}}{2d}\right)a_1
    +\left(\frac{\varepsilon_{I}}{4d}+\frac{\varepsilon_{I_{+}}}{2d}\right)a_2
    +\left(-\frac{\varepsilon_{I}}{4d}-\frac{\varepsilon_{I_{-}}}{2d}\right)a_3
    +\left(-\frac{\varepsilon_{I}}{4d}+\frac{\varepsilon_{I_{-}}}{2d}\right)a_4.
\end{equation*}
That means, $\wt{b}_{I}$ is zero outside $I$ and
\begin{gather*}
    \wt{b}_{I}|_{I_{++}}=-\frac{1}{4d}a_1+\frac{3}{4d}a_2-\frac{1}{4d}a_3-\frac{1}{4d}a_4,\\
    \wt{b}_{I}|_{I_{+-}}=\frac{3}{4d}a_1-\frac{1}{4d}a_2-\frac{1}{4d}a_3-\frac{1}{4d}a_4,\\
    \wt{b}_{I}|_{I_{-+}}=-\frac{1}{4d}a_1-\frac{1}{4d}a_2-\frac{1}{4d}a_3+\frac{3}{4d}a_4,\\
    \wt{b}_{I}|_{I_{--}}=-\frac{1}{4d}a_1-\frac{1}{4d}a_2+\frac{3}{4d}a_3-\frac{1}{4d}a_4.
\end{gather*}
So, setting
\begin{equation*}
    \wt{a}_{J}:=\wt{a}_{I}+\wt{b}_{I}|_{J}
\end{equation*}
for each $J\in\mathrm{ch}^2(J)$, we deduce
\begin{gather*}
    \lambda_{i,I_{++}}=\lambda_{i,I}-\frac{1}{4d},~i=1,3,4,\quad
    \lambda_{2,I_{++}}=\lambda_{2,I}+\frac{3}{4d},\\
    \lambda_{i,I_{+-}}=\lambda_{i,I}-\frac{1}{4d},~i=2,3,4,\quad
    \lambda_{1,I_{+-}}=\lambda_{1,I}+\frac{3}{4d},\\
    \lambda_{i,I_{-+}}=\lambda_{i,I}-\frac{1}{4d},~i=1,2,3,\quad
    \lambda_{4,I_{-+}}=\lambda_{4,I}+\frac{3}{4d},\\
    \lambda_{i,I_{--}}=\lambda_{i,I}-\frac{1}{4d},~i=1,2,4,\quad
    \lambda_{3,I_{--}}=\lambda_{3,I}+\frac{3}{4d}.
\end{gather*}
In view of \eqref{eq:barycentric_cooordinates_inside_simplex}, our definition makes sense.

We repeat this step for each $J\in\mathrm{ch}^2(I)$ such that $\wt{a}_{J}$ still lies in the interior of the simplex. If it happens that $\wt{a}_{J}$ lies on some face of the simplex, then a different algorithm will be used in the sequel, which we describe below.

We call the intervals $J\in\cD_{\mathrm{o}}(I_0)$ for which the above algorithm gives $\wt{a}_J$ in the \emph{interior of a face} of the simplex \emph{stopping intervals of order 1}. The intervals yielding a point in the \emph{interior of an edge} are called \emph{stopping intervals of order 2}. Finally, the intervals reaching directly a vertex are termed just \emph{stopping intervals}. Observe that by applying Lemma~\ref{lem:finiteness_hitting_times} (with $\{\omega_n\}$ being the sequence of the successive increments of any of the four barycentric coordinates) we have that for a.e.~$x\in I_0$ there exists $I\in\cD_{\mathrm{o}}(I_0)$, with $x\in I$, such that $\wt{a}_{I}$ lies on a face of the $3$-simplex.

\subsubsection{Inside a face}

Assume now that $I$ is a stopping interval of order 1 such that $\wt{a}_{I}\in\mathrm{Conv}(a_i,a_j,a_k)$ for some $i,j,k\in\{1,2,3,4\}$ with $i<j<k$. Observe that
\begin{equation}
    \label{eq:barycentric_coordinates_inside_face}
    \lambda_{i,I},\lambda_{j,I},\lambda_{k,I}\geq\frac{1}{4d},\quad \lambda_{i,I}+\lambda_{j,I}+\lambda_{k,I}=1.
\end{equation}
Let $\hat{I}$ be the dyadic parent of $I$, that is the unique dyadic subinterval of $I_0$ containing $I$ with length $2\ell(I)$. If $I=(\hat{I})_{+}$, then we define
\begin{equation*}
    \wt{b}_{I}:=\frac{a_i-a_j}{4d}\varepsilon_{I}.
\end{equation*}
If $I=(\hat{I})_{-}$, then we define
\begin{equation*}
    \wt{b}_{I}:=\frac{a_i-a_k}{4d}\varepsilon_{I}.
\end{equation*}
We also define
\begin{equation*}
    \wt{a}_{J}:=\wt{a}_{I}+\wt{b}_{I}|_{J}\text{ for each }J\in\mathrm{ch}^2(I).
\end{equation*}
Concretely: if $I=(\hat{I})_{+}$, then
\begin{gather*}
    \lambda_{i,I_{++}}=\lambda_{i,I_{+-}}=\lambda_{i,I}+\frac{1}{4d},\quad \lambda_{i,I_{-+}}=\lambda_{i,I_{--}}=\lambda_{i,I}-\frac{1}{4d},\\
    \lambda_{j,I_{++}}=\lambda_{j,I_{+-}}=\lambda_{j,I}-\frac{1}{4d},\quad
    \lambda_{j,I_{-+}}=\lambda_{j,I_{--}}=\lambda_{j,I}+\frac{1}{4d},\\
    \lambda_{k,I_{++}}=\lambda_{k,I_{+-}}=
    \lambda_{k,I_{-+}}=\lambda_{k,I_{--}}=\lambda_{k,I},
\end{gather*}
whereas if $I=(\hat{I})_{-}$, then
\begin{gather*}
    \lambda_{i,I_{++}}=\lambda_{i,I_{+-}}=\lambda_{i,I}+\frac{1}{4d},\quad
    \lambda_{i,I_{-+}}=\lambda_{i,I_{--}}=\lambda_{i,I}-\frac{1}{4d},\\
    \lambda_{j,I_{++}}=\lambda_{j,I_{+-}}=
    \lambda_{j,I_{-+}}=\lambda_{j,I_{--}}=\lambda_{j,I},\\
    \lambda_{k,I_{++}}=\lambda_{k,I_{+-}}=\lambda_{k,I}-\frac{1}{4d},\quad
    \lambda_{k,I_{-+}}=\lambda_{k,I_{--}}=\lambda_{k,I}+\frac{1}{4d}.
\end{gather*}
This step is then repeated in each $J\in\mathrm{ch}^2(I)$ such that $\wt{a}_{J}$ still lies in the interior of the face. If $\wt{a}_{J}$ lies on the interior of an edge, then $J$ is a stopping interval of order 2 and one moves on to the third part of the algorithm. If $\wt{a}_{J}$ lies on a vertex, then $J$ is just a stopping interval.

Again, by an application of Lemma~\ref{lem:finiteness_hitting_times} we have that for a.e.~$x\in I_0$ there exists $I\in\cD_{\mathrm{o}}(I_0)$, with $x\in I$, such that $\wt{a}_{I}$ lies on an edge of the $3$-simplex.

\subsubsection{Inside an edge}
\label{subsub:small_step_edge}

Let $I$ be a stopping interval of order 2 with $\wt{a}_{I}\in\mathrm{Conv}(a_i,a_j)$ for some $i,j\in\{1,2,3,4\}$ with $i<j$. Observe that
\begin{equation*}
    \lambda_{i,I},\lambda_{j,I}\geq\frac{1}{4d},\quad \lambda_{i,I}+\lambda_{j,I}=1.
\end{equation*}
We define
\begin{equation*}
    \wt{b}_{I}:=\frac{a_j-a_i}{4d}\varepsilon_{I}
\end{equation*}
and
\begin{equation*}
    \wt{a}_{J}:=\wt{a}_{I}+\wt{b}_{I}|_{J},\quad J\in\mathrm{ch}^2(I).
\end{equation*}
This means
\begin{gather*}
    \lambda_{i,I_{++}}=\lambda_{i,I_{+-}}=\lambda_{i,I}-\frac{1}{4d},\quad
    \lambda_{i,I_{-+}}=\lambda_{i,I_{--}}=\lambda_{i,I}+\frac{1}{4d}\\
    \lambda_{j,I_{++}}=\lambda_{j,I_{+-}}=\lambda_{j,I}+\frac{1}{4d},\quad
    \lambda_{j,I_{-+}}=\lambda_{j,I_{--}}=\lambda_{j,I}-\frac{1}{4d}.
\end{gather*}
This step is repeated in each $J\in\mathrm{ch}^2(I)$ such that $\wt{a}_{J}$ still lies in the interior of the edge. If $\wt{a}_{J}$ lies on a vertex, then $J$ is a stopping interval.

A final application of Lemma~\ref{lem:finiteness_hitting_times} shows that for a.e.~$x\in I_0$ there exists $I\in\cD_{\mathrm{o}}(I_0)$, with $x\in I$, such that $\wt{a}_{I}$ lies on a vertex of the simplex.

\subsubsection{On a vertex}

If $I$ is a stopping interval, then we just define
\begin{equation*}
    \wt{b}_{J}:=0,\quad\wt{a}_{J}:=\wt{a}_{I},
\end{equation*}
for all $J\in\cD_{\mathrm{o}}(I)$ with $J\subseteq I$.

\subsubsection{Limit function}

As noted above, the stochastic process $\{\wt{a}_n\}^{\infty}_{n=0}$ we described is automatically a $\mathbb{F}$-martingale, thanks to the properties of the differences $\wt{b}_n$. We already observed above that it has a limit function $\wt{a}$, such that for almost every $x\in I_0$ there is $n\in\N$ with $\wt{a}_k(x)=\wt{a}(x)\in\{a_1,a_2,a_3,a_4\}$ for all $k=n,n+1,\ldots$. Finally, Lemma~\ref{lem:finiteness_hitting_times} yields $\mathrm{m}(\{\wt{a}=a_i\})=\frac{1}{4}\ell(I_0)$ for each $i=1,2,3,4$.

\subsection{Getting the ``damage''}
\label{subs:one_block_damage}

In this subsection we show that the random walk described above satisfies a certain estimate. The significance of it will become apparent in the next section.

We consider some dyadic interval $K$ and points
\begin{equation*}
    (x_{J},y_{J})\in\R^n\times\R^n,\quad J\in\{K,K_{\pm},K_{\pm\pm}\}
\end{equation*}
satisfying the (martingale) relations
\begin{equation*}
    x_{J}=\frac{1}{2}(x_{J_{+}}+x_{J_{-}}),\quad y_{J}=\frac{1}{2}(y_{J_{+}}+y_{J_{-}}),\quad J\in\{K,K_{\pm}\}.
\end{equation*}
Let us consider the convex set $C$ generated by the points $(x_{J},y_{J})$, $J\in\{K_{\pm\pm}\}$. Then, there is a unique affine map $A:\mathscr{K}\to C$ from the $3$-simplex $\mathscr{K}$ onto $C$ such that 
\begin{gather*}
    A(a_1)=(x_{K_{+-}},y_{K_{+-}}),\quad A(a_2)=(x_{K_{++}},y_{K_{++}}),\\
    A(a_3)=(x_{K_{--}},y_{K_{--}}),\quad A(a_4)=(x_{K_{-+}},y_{K_{-+}}).
\end{gather*}
Thus, we can consider the $\mathbb{F}$-adapted martingale $\{\wt{F}_n=A\circ\wt{a}_n\}_{n=0}^{\infty}$ with limit function $\wt{F}=A\circ\wt{a}$. Write $\wt{F}=(\wt{f},\wt{g})$. We show here that
\begin{align}
    \label{eq:damage_on_each_simplex}
    (\mathcal{H}^{\mathrm{dy}}\wt f,\wt{g})_{L^2(I_0)}&=\frac{\mathbb{E}[\tau]}{d^2}|I_0|\bigg(\frac{1}{2}\langle\Delta_{K}x,\Delta_{K_{+}}y-\Delta_{K_{-}}y\rangle
    +\frac{1}{2}\langle\Delta_{K_{+}}x-\Delta_{K_{-}}x,\Delta_{K}y\rangle\\
    \nonumber&+\frac{1}{4}(\langle\Delta_{K_{+}}x,\Delta_{K_{-}}y\rangle-\langle\Delta_{K_{-}}x,\Delta_{K_{+}}y\rangle)\bigg),
\end{align}
where we denote $\Delta_{J}z=z_{J_{+}}-z_{J_{-}}$ and we consider the $\mathbb{F}$-stopping time $\tau$ defined by
\begin{equation*}
    \tau=\inf\{n\geq0:~\wt{a}_n\text{ lies on the boundary of }\mathscr{K}\}.
\end{equation*}
 We emphasize that $L^2(I_0)$ refers to non-normalized Lebesgue measure on $I_0$. Essentially, \eqref{eq:damage_on_each_simplex} is a consequence of the \emph{discrete It{\^o} isometry}. Here, we present an explicit argument, reproving a form of the latter in the present special case.

To see \eqref{eq:damage_on_each_simplex}, we begin by writing out the definition of $\mathcal{H}^{\mathrm{dy}}$ and expanding the left hand side in \eqref{eq:damage_on_each_simplex}, getting
\begin{align}
    \label{eq:write_out_damage_on_each_simplex}
    (\mathcal{H}^{\mathrm{dy}}\wt f,\wt{g})_{L^2(I_0)}&=\sum_{I\in\cD_{\mathrm{o}}(I_0)}\frac{1}{2}\langle\Delta_{I}\wt{f},\Delta_{I_{+}}\wt{g}-\Delta_{I_{-}}\wt{g}\rangle|I|\\
    \nonumber&+\sum_{I\in\cD_{\mathrm{o}}(I_0)}\frac{1}{2}\langle\Delta_{I_{+}}\wt{f}-\Delta_{I_{-}}\wt{f},\Delta_{I}\wt{g}\rangle|I|\\
    \nonumber&+\sum_{I\in\cD_{\mathrm{o}}(I_0)}\frac{1}{4}(\langle\Delta_{I_{+}}\wt{f},\Delta_{I_{-}}\wt{g}\rangle-\langle\Delta_{I_{-}}\wt{f},\Delta_{I_{+}}\wt{g}\rangle)|I|,
\end{align}
where we denote $\Delta_{J}h=\langle h\rangle_{J_{+}}-\langle h\rangle_{J_{-}}$. Let us denote by $\mathcal{S}$ the family of all $I\in\cD_{\mathrm{o}}(I_0)$ such that $\wt{a}_{I}$ lies in the interior of the $3$-simplex $\mathscr{K}$. Then, by construction of the martingale $\{\wt{a}_n\}$ we have $\Delta_{I_{\pm}}\wt{f}=\Delta_{I_{\pm}}\wt{g}=0$, for every $I\in\cD_{\mathrm{o}}(I_0)\setminus\mathcal{S}$. Moreover, for every $I\in\mathcal{S}$ we have, again by construction of $\{\wt{a}_n\}$,
\begin{equation*}
    \Delta_{I}\wt{f}=\frac{1}{d}\Delta_{K}x,~\Delta_{I_{\pm}}\wt{f}=\frac{1}{d}\Delta_{K_{\pm}}x\quad\text{and}\quad\Delta_{I}\wt{g}=\frac{1}{d}\Delta_{K}y,~\Delta_{I_{\pm}}\wt{g}=\frac{1}{d}\Delta_{K_{\pm}}y.
\end{equation*}
Thus, \eqref{eq:write_out_damage_on_each_simplex} reduces to
\begin{align}
    \label{eq:damage_on_each_simplex_intervals}
    (\mathcal{H}^{\mathrm{dy}}\wt f,\wt{g})_{L^2(I_0)}&=\frac{1}{d^2}\bigg(\frac{1}{2}\langle\Delta_{K}x,\Delta_{K_{+}}y-\Delta_{K_{-}}y\rangle
    +\frac{1}{2}\langle\Delta_{K_{+}}x-\Delta_{K_{-}}x,\Delta_{K}y\rangle\\
    \
    \nonumber&+\frac{1}{4}(\langle\Delta_{K_{+}}x,\Delta_{K_{-}}y\rangle-\langle\Delta_{K_{-}}x,\Delta_{K_{+}}y\rangle)\bigg)\sum_{I\in\mathcal{S}}|I|.
\end{align}
Now observe that
\begin{equation*}
    \sum_{I\in\mathcal{S}}|I|=|I_0|\cdot\mathbb{E}\bigg[\sum_{I\in\mathcal{S}}\mathbf{1}_{I}\bigg]
\end{equation*}
and that obviously $\sum_{I\in\mathcal{S}}\mathbf{1}_{I}=\tau$, proving \eqref{eq:damage_on_each_simplex}.

It is essential in \eqref{eq:damage_on_each_simplex} that the factor one obtains in front of the right hand side is independent of $x,y,J$. It is equally essential to check at this point that $\mathbb{E}[\tau]\geq cd^2$ for some absolute constant $c>0$. To see this, we consider the $\R^4$-valued stochastic process $X=\{X_{n}\}_{n=0}^{\infty}$ collecting the barycentric coordinates of $\{\wt{a}_n\}^{\infty}_{n=0}$. Note that $X$ is still a $\mathbb{F}$-adapted martingale. Then, it holds
\begin{equation*}
    \tau=\inf\{n\geq0:~X_n\text{ has a zero coordinate}\}.
\end{equation*}
Observing that $|X_n-X_{n-1}|^2\leq\frac{3}{2d^2}$, we obtain
\begin{equation*}
    \mathbb{E}[|X_n-X_{n-1}|^2|\cF_n]\leq\frac{3}{2d^2}.
\end{equation*}
Thus, the process $\bigg\{|X_n-X_0|^2-n\frac{3}{2d^2}\bigg\}$ is a $\mathbb{F}$-supermartingale. An application of the optional stopping theorem gives
\begin{equation*}
    \mathbb{E}\bigg[|X_{\tau}-X_0|^2-\tau\frac{3}{2d^2}\bigg]\leq \mathbb{E}\bigg[|X_{0}-X_0|^2-0\cdot d\sqrt{\frac{3}{2}}\bigg]=0,
\end{equation*}
in other words $\mathbb{E}[\tau]\geq\frac{2d^2}{3}\mathbb{E}[|X_{\tau}-X_0|^2]$. It is clear that $|X_{\tau}-X_{0}|\geq\frac{1}{2}$, for at each point $x$ with $\tau(x)<\infty$, $X_{\tau}(x)$ has a zero coordinate, and $X_{0}\equiv\bigg(\frac{1}{4},\frac{1}{4},\frac{1}{4},\frac{1}{4}\bigg)$. So $\mathbb{E}[\tau]\geq\frac{d^2}{6}$.

Let us note that a similar argument using a lower bound for $|X_n-X_{n-1}|^2$ and a submartingale shows that $\mathbb{E}[\tau]\leq 4d^2\mathbb{E}[|X_{\tau}-X_0|^2]$. Since $|X_{\tau}-X_{0}|\leq 2$, we also get $\mathbb{E}[\tau]\leq 8d^2$.

\section{``Small step'' examples}
\label{sec:small_step_examples}

Let us recall the bounded function $F=(W,V,\mathbf{f},\mathbf{g})$ on $I^{0}$ taking values in the space $\mathscr{X}:=\R^{2\times 2}\times\R^{2\times 2}\times\R^2\times\R^2$ from Section~\ref{sec:large_step_examples}. Using the ``small step'' random walks from Section~\ref{sec:small_step_random_walks_tetrahedron} as building blocks, we will construct a new bounded $\mathscr{X}$-valued function $\wt{F}=(\wt{W},\wt{V},\wt{\mathbf{f}},\wt{\mathbf{g}})$ on $I^{0}$, such that there is a measure preserving map $T:I^0\to I^0$ with $\wt{F}=F\circ T$ almost everywhere, $(\mathcal{H}^{\mathrm{dy}}\wt{\mathbf{f}},\wt{\mathbf{g}})_{L^2}=c(\mathcal{H}^{\mathrm{dy}}\mathbf{f},\mathbf{g})_{L^2}$ for some absolute positive constant $c$, and in addition $[\wt{W}]^{\mathrm{dy}}_{\mathbf{A}_2}\sim Q$ and $S^{\mathrm{dy}}_{\wt{W}},S^{\mathrm{dy}}_{\wt{V}}\leq 1+\varepsilon$.

\subsection{The iterative construction}

Instead of directly defining the function $\wt{F}$, we will describe explicitly all averages $\{\langle\wt{F}\rangle_{I}\}_{I\in\cD_{\mathrm{o}}}$. Observe that the averages over even dyadic intervals are then uniquely determined by the usual martingale relations.

Before laying out the details, we mention a couple of important aspects of the construction:
\begin{itemize}
    \item To each $x\in[0,1)$ there will correspond a trajectory in the space $\mathscr{X}$. However, instead of directly thinking of points, we will be thinking of the sequence of the dyadic intervals containing them.
    
    \item Each trajectory will be defined recursively. Each step of this recursion will be taking place inside an odd dyadic interval, that is an interval in $\cD_{\mathrm{o}}$.

    \item As explained in Section~\ref{sec:large_step_examples}, to each interval $K\in\mathscr{F}$ there corresponds a (possibly degenerate) tetrahedron $\mathscr{C}(K)$, such that the old average $\langle F\rangle_{K}$ occupies the center of that tetrahedron. Almost every recursive step will correspond to a single small step random walk of the type described in Section~\ref{sec:small_step_random_walks_tetrahedron} on the tetrahedron corresponding to some interval in $\mathscr{F}$
    
    For each $K\in\mathscr{F}$ we will be denoting by $\mathscr{J}(K)$ the set of all dyadic intervals, in which a recursive step takes place on the tetrahedron corresponding to $K$. This amounts to the points/trajectories visiting $\langle F\rangle_{K}$ in the course of the construction.
\end{itemize}

\medskip

\textbf{Initialization.} We initialize the construction by defining
\begin{equation*}
    \langle\wt{F}\rangle_{I^{0}_{+}}:=\langle F\rangle_{I^0}\quad\text{and}\quad\langle\wt{F}\rangle_{I^{0}_{-}}:=\langle F\rangle_{I^0}.
\end{equation*}
It might seem surprising that for this one step at the very beginning our small step random walk does not move at all. However, as already mentioned above, each recursive step will take place inside an odd dyadic interval. The very first odd dyadic intervals are $I^{0}_{+}$ and $I^{0}_{-}$. Thus, it is reasonable to put the new averages for those at the very first old average.
\medskip

\textbf{1st step.} The first step of the iteration is unique to the dyadic children of $I^{0}$, and differs from what we call below regular iterative step.

Let $J\in\{I^{0}_{+},I^{0}_{-}\}$. Then, we perform the one dimensional small step 4-adic random walk in $J$ on the segment with endpoints $\langle F\rangle_{I^{0}_{\pm}}$, just as in \ref{subsub:small_step_edge}, starting from the midpoint $\langle F\rangle_{I^{0}}$ of this segment. In this way, further averages of $\wt{F}$ are obtained, and we stop ``temporarily'' once we reach one of the two vertices of the segment. As noted in Section~\ref{sec:large_step_examples}, this segment is by construction non-degenerate, i.e.~$\langle F\rangle_{I^{0}_{+}}\neq\langle F\rangle_{I^{0}_{-}}$. For each $K\in\{I^{0}_{+},I^{0}_{-}\}$, we let $\mathscr{J}(J,K)$ be the family of the stopping intervals $I\in\cD_{\mathrm{o}}(J)$ with $\langle\wt{F}\rangle_{I}=\langle F\rangle_{K}$, which amounts to the pieces of trajectories of points in $J$ reaching the old average $\langle F\rangle_{K}$.

Finally, for each $K\in\{I^{0}_{+},I^{0}_{-}\}$ we set $\mathscr{J}(K):=\mathscr{J}(I^{0}_{+},K)\cup\mathscr{J}(I^{0}_{-},K)$. This corresponds to the pieces of trajectories of all points in $I^{0}$ that have reached the old average $\langle F\rangle_{K}$ after this special first step is completed. Observe that the intervals in $\mathscr{J}(K)$ are pairwise disjoint and that
\begin{equation*}
    \sum_{I\in\mathscr{J}(K)}|I|=\sum_{I\in\mathscr{J}(I^{0}_{+},K)}|I|+\sum_{I\in\mathscr{J}(I^{0}_{-},K)}|I|=\frac{|I^0_{+}|}{2}+\frac{|I^0_{-}|}{2}=|K|,
\end{equation*}
where in the second $=$ we used that we stop with equal probability on each endpoint of the segment, as explained in Section~\ref{sec:small_step_random_walks_tetrahedron}. Thus, the new small step random walk reaches $\langle F\rangle_{K}$ with the same probability as the old large step one.

\medskip

\textbf{Regular iterative step.} Now we come to the description of the regular iterative step, which will be applied from now on.

Assume now that for some $K\in\mathscr{F}$ we have defined the family $\mathscr{J}(K)$ of pairwise disjoint odd dyadic subintervals of $I^0$. We recall that in intuitive terms this is just the family of pieces of trajectories having reached the old average $\langle F\rangle_{K}$ after all the steps that have already been completed. In symbols, this reads as $\langle\wt{F}\rangle_{I}=\langle F\rangle_{K}$ for all $I\in\mathscr{J}(K)$. The aforementioned old average is the center of mass of the tetrahedron $\mathscr{C}(K)$. Intuitively speaking, the small step random walk has entered a new tetrahedron. We assume also that the new small step random walk has reached $\langle F\rangle_{K}$ with the same probability as the old large step one, i.e.~$\sum_{I\in\mathscr{J}(K)}|I|=|K|$.

Likewise, we assume that in the steps that have been so far completed we have already defined all the averages $\langle\wt{F}\rangle_{J}$ over all odd dyadic intervals $J$ containing (not necessarily strictly) some interval $I\in\mathscr{J}(K)$.

Then, in each $I\in\mathscr{J}(K)$ we perform the small step 4-adic random walk on the tetrahedron $\mathscr{C}(K)$ corresponding to $K$, starting from the center of mass $\langle F\rangle_{K}$. In this way, further averages of $\wt{F}$ are obtained, and we stop ``temporarily'' once we reach one of the four vertices of the tetrahedron. Of course, since the tetrahedron $\mathscr{C}(K)$ might be degenerate, one actually considers such a random walk on some non-degenerate tetrahedron, say in $\R^3$, which through an affine map is mapped onto a random walk on $\mathscr{C}(K)$, as described in Subsection~\ref{subs:one_block_damage}. For each $L\in\mathscr{F}$ with $L\subsetneq K$ (if any) we denote by $\mathscr{J}(I,L)$ the family of the stopping intervals $M\in\cD_{\mathrm{o}}(I)$ such that $\langle\wt{F}\rangle_{M}=\langle F\rangle_{L}$.

Finally, we set $\mathscr{J}(L):=\bigcup_{I\in\mathscr{J}(K)}\mathscr{J}(I,L)$. This is just the family of all pieces of trajectories having reached the old average $\langle F\rangle_{L}$ after all the previous steps and the current new step have been completed. Observe that the intervals in $\mathscr{J}(L)$ are odd, pairwise disjoint and satisfy
\begin{equation*}
    \sum_{M\in\mathscr{J}(L)}|M|=\sum_{I\in\mathscr{J}(K)}\sum_{M\in\mathscr{J}(I,L)}|M|=\sum_{I\in\mathscr{J}(K)}\frac{|I|}{4}=\frac{|K|}{4}=|L|,
\end{equation*}
where in the second $=$ we used that we stop with equal probability on each endpoint of the tetrahedron, as explained in Section~\ref{sec:small_step_random_walks_tetrahedron}. Thus, the new small step random walk reaches $\langle F\rangle_{L}$ with the same probability as the old large step one.

For each each $I\in\mathscr{J}(K)$ and $L\in\mathscr{E}$ with $L\subseteq K$ we denote by $\mathscr{I}(I,L)$ the family of the stopping intervals $M\in\cD_{\mathrm{o}}(I)$ such that $\langle\wt{F}\rangle_{M}=\langle F\rangle_{L}$. We set $\mathscr{I}(L):=\bigcup_{I\in\mathscr{J}(K)}\mathscr{I}(I,L)$. As before, the intervals in $\mathscr{J}(L)$ are odd, pairwise disjoint and satisfy $\sum_{M\in\mathscr{J}(L)}|M|=|M|$. We let $\wt{F}$ being identically equal to $\langle\wt{F}\rangle_{M}$ on each such interval $M$. This completes the inductive description.

\medskip

For any intervals $I,J$, denote by $\psi_{I,J}:I\to J$ the unique orientation-preserving affine map mapping $I$ onto $J$. Observe that the family $\bigcup_{L\in\mathscr{E}}\mathscr{J}(L)$ forms a partition of $I^{0}$, up to a set of zero measure. Thus, one can consider a map $T:I^{0}\to I^{0}$ which for almost every $x\in I^{0}$ is given by
\begin{equation*}
    T(x):=\psi_{M,L}(x)\quad\text{ if }x\in M\text{ for some }M\in\mathscr{I}(L)\text{ and }L\in\mathscr{E}.
\end{equation*}
Since $\sum_{M\in\mathscr{J}(L)}|M|=|L|$, for all $L\in\mathscr{E}$, we conclude that $T$ is measure-preserving. Finally, in view of \eqref{eq:F_decomposition}, since
\begin{equation}
    \label{eq:F_tilde_decomposition}
    \wt{F}=\sum_{L\in\mathscr{E}}\sum_{M\in\mathscr{I}(L)}\langle\wt{F}\rangle_{M}\mathbf{1}_{M}=\sum_{L\in\mathscr{E}}\langle F\rangle_{L}\sum_{M\in\mathscr{I}(L)}\mathbf{1}_{M}
\end{equation}
we deduce $\wt{F}=F\circ T$ almost everywhere.

Let us observe that in particular $\wt{V}=\wt{W}^{-1}$ a.e.~and
\begin{equation*}
    |\mathbf{f}\|_{L^2(W)}=\|\wt{\mathbf{f}}\|_{L^2(\wt{W})},\quad |\mathbf{g}\|_{L^2(W^{-1})}\|\wt{\mathbf{g}}\|_{L^2(\wt{W}^{-1})}.
\end{equation*}

\subsection{Preservation of the damage}

Here we estimate $(\mathcal{H}^{\text{dy}}\wt{\mathbf{f}},\wt{\mathbf{g}})_{L^2(I^0)}$. We begin by writing
\begin{align}
    \label{eq:write_out_total_damage}
    (\mathcal{H}^{\mathrm{dy}}\wt{\mathbf{f}},\wt{\mathbf{g}})_{L^2(I^0)}&=\sum_{I\in\cD_{\mathrm{o}}}\frac{1}{2}\langle\Delta_{I}\wt{\mathbf{f}},\Delta_{I_{+}}\wt{\mathbf{g}}-\Delta_{I_{-}}\wt{\mathbf{g}}\rangle|I|\\
    \nonumber&+\sum_{I\in\cD_{\mathrm{o}}}\frac{1}{2}\langle\Delta_{I_{+}}\wt{\mathbf{f}}-\Delta_{I_{-}}\wt{\mathbf{f}},\Delta_{I}\wt{\mathbf{g}}\rangle|I|\\
    \nonumber&+\sum_{I\in\cD_{\mathrm{o}}}\frac{1}{4}(\langle\Delta_{I_{+}}\wt{\mathbf{f}},\Delta_{I_{-}}\wt{\mathbf{g}}\rangle-\langle\Delta_{I_{-}}\wt{\mathbf{f}},\Delta_{I_{+}}\wt{\mathbf{g}}\rangle)|I|.
\end{align}
Observe that for every odd dyadic interval $I$ one of the following holds:
\begin{itemize}
    \item[(a)] There is no $K\in\mathscr{F}$ such that $I\subseteq J$ for some $J\in\mathscr{I}(K)$. We write $I\in\mathscr{A}$ in this case.

    \item[(b)] There is some $K\in\mathscr{F}$ such that $I\subseteq J$ for some $J\in\mathscr{I}(K)$. In this case we denote by $\mathscr{T}(I)$ the smallest such interval $K$ and by $\mathscr{S}(K)$ the corresponding interval $J$.
\end{itemize}
Thus, we can write $(\mathcal{H}^{\mathrm{dy}}\wt{\mathbf{f}},\wt{\mathbf{g}})_{L^2(I^0)}=A_1+A_2$, where
\begin{align*}
    A_1&:=\sum_{I\in\mathscr{A}}\frac{1}{2}\langle\Delta_{I}\wt{\mathbf{f}},\Delta_{I_{+}}\wt{\mathbf{g}}-\Delta_{I_{-}}\wt{\mathbf{g}}\rangle|I|\\
    \nonumber&+\sum_{I\in\mathscr{A}}\frac{1}{2}\langle\Delta_{I_{+}}\wt{\mathbf{f}}-\Delta_{I_{-}}\wt{\mathbf{f}},\Delta_{I}\wt{\mathbf{g}}\rangle|I|\\
    \nonumber&+\sum_{I\in\mathscr{A}}\frac{1}{4}(\langle\Delta_{I_{+}}\wt{\mathbf{f}},\Delta_{I_{-}}\wt{\mathbf{g}}\rangle-\langle\Delta_{I_{-}}\wt{\mathbf{f}},\Delta_{I_{+}}\wt{\mathbf{g}}\rangle)|I|
\end{align*}
and
\begin{align*}
    A_2&:=\sum_{K\in\mathscr{F}}\bigg[\sum_{\mathscr{T}(I)=K}\frac{1}{2}\langle\Delta_{I}\wt{\mathbf{f}},\Delta_{I_{+}}\wt{\mathbf{g}}-\Delta_{I_{-}}\wt{\mathbf{g}}\rangle|I|\\
    \nonumber&+\sum_{\mathscr{T}(I)=K}\frac{1}{2}\langle\Delta_{I_{+}}\wt{\mathbf{f}}-\Delta_{I_{-}}\wt{\mathbf{f}},\Delta_{I}\wt{\mathbf{g}}\rangle|I|\\
    \nonumber&+\sum_{\mathscr{T}(I)=K}\frac{1}{4}(\langle\Delta_{I_{+}}\wt{\mathbf{f}},\Delta_{I_{-}}\wt{\mathbf{g}}\rangle-\langle\Delta_{I_{-}}\wt{\mathbf{f}},\Delta_{I_{+}}\wt{\mathbf{g}}\rangle)|I|\bigg].
\end{align*}

For $A_1$, observe that if $I\in\mathscr{A}$, then $\langle\wt{F}\rangle_{I}$ lies still on the interior of the initial segment, thus by construction $\Delta_{I_{\pm}}\wt{F}=0$. It follows that $A_2=0$.

To estimate $A_2$, observe that $\mathscr{T}(I)=K$, then $I$ is one of the dyadic subintervals of $J$ that appear in the random walk in $J=\mathscr{S}(I)$ on the simplex $\mathscr{C}(K)$. Thus, by the computations in Subsection~\ref{subs:one_block_damage} we obtain
\begin{align*}
    A_2&=\sum_{K\in\mathscr{F}}\sum_{J\in\mathscr{I}(K)}\frac{c}{d^2}|J|\bigg(\frac{1}{2}\langle\Delta_{J}\mathbf{f},\Delta_{K_{+}}\mathbf{g}-\Delta_{K_{-}}\mathbf{g}\rangle
    +\frac{1}{2}\langle\Delta_{K_{+}}\mathbf{f}-\Delta_{K_{-}}\mathbf{f},\Delta_{K}\mathbf{g}\rangle\\
    &+\frac{1}{4}(\langle\Delta_{K_{+}}\mathbf{f},\Delta_{K_{-}}\mathbf{g}\rangle-\langle\Delta_{K_{-}}\mathbf{f},\Delta_{K_{+}}\mathbf{g}\rangle)\bigg),
\end{align*}
where the constant $c>0$ equals $\mathbb{E}[\tau]$ in the notation of Subsection~\ref{subs:one_block_damage}. Observe that this constant $c$ is actually universal, that is independent of both $J$ and $K$, because
\begin{itemize}
    \item dyadic intervals equipped with normalized Lebesgue measure and the corresponding dyadic filtration form isomorphic filitered probability spaces, and

    \item the random walk that underlies this construction is purely expressible in terms of barycentric coordinates and has nothing to do with the concrete tetrahedron one uses, as apparent when estimating $\mathbb{E}[\tau]$ in Subsection~\ref{subs:one_block_damage}.
\end{itemize}
Recalling that $\sum_{J\in\mathscr{I}(K)}\frac{c}{d^2}|J|=|K|$ and Subsection~\ref{subs:damage_large_step}, we obtain
\begin{equation*}
    (\mathcal{H}^{\text{dy}}\wt{\mathbf{f}},\wt{\mathbf{g}})_{L^2(I^0)}=\frac{c}{d^2}(\mathcal{H}^{\text{dy}}\mathbf{f},\mathbf{g})_{L^2(I^0)}.
\end{equation*}
We saw in Subsection~\ref{subs:one_block_damage} that $c\gtrsim d^2$, therefore
\begin{equation*}
    |(\mathcal{H}^{\text{dy}}\wt{\mathbf{f}},\wt{\mathbf{g}})_{L^2(I^0)}|\gtrsim Q^{3/2}|\mathbf{f}\|_{L^2(W)}\|\mathbf{g}\|_{L^2(W^{-1})}
    =Q^{3/2}\|\wt{\mathbf{f}}\|_{L^2(\wt{W})}\|\wt{\mathbf{g}}\|_{L^2(\wt{W}^{-1})}.
\end{equation*}

\subsection{Checking the matrix Muckenhoupt characteristic and the dyadic smoothness}
\label{subs:small_step_dyadic_smoothness}

We will first need the following elementary functional theoretical observation. The proof is standard, but we include it for the reader's convenience. Even though we need it only for matrices, it can be done without additional difficulties for operators on any Hilbert space. If $A$ is a bounded, self-adjoint, linear operator on a Hilbert space $H$, then we denote $m(A):=\min\sigma(A)$.

\begin{lem}
    \label{lem:near_convex_comparable}
    Let $n$ be a positive integer and let $A_1,\ldots,A_n$ be bounded, positive definite, linear operators on a Hilbert space $H$.
    \begin{enumerate}
        \item We have
        \begin{equation*}
            \inf\left\{m\left(\sum_{i=1}^{n}\lambda_iA_i\right):~\lambda_i\in[0,1],~\sum_{i=1}^{n}\lambda_i=1\right\}\geq\min_{i=1,\ldots,n}m(A_i)>0.
        \end{equation*}

        \item Let $\e>0$. Set
        \begin{equation*}
            c:=\frac{\min_{i=1,\ldots,n}m(A_i)}{\sum_{i=1}^{n}\Vert A_i\Vert}.
        \end{equation*}
        Let $\lambda_i,\mu_i\in[0,1]$ with $\sum_{i=1}^{n}\lambda_i=\sum_{i=1}^{n}\mu_i=1$ and $|\lambda_i-\mu_i|\leq\e c$. Set $C:=\sum_{i=1}^{n}\lambda_iA_i$ and $D:=\sum_{i=1}^{n}\mu_iA_i$. Then, we have $C\leq(1+\e)D$ and $D\leq(1+\e)C$.
    \end{enumerate}
\end{lem}

\begin{rem}
    It is easy to see that if $K$ is any set of bounded, self-adjoint, linear operators on $H$ that is compact in the norm operator topology, then $\inf\{m(A):~A\in K\}$ is actually attained. In part 1 of Lemma~\ref{lem:near_convex_comparable} we chose to give a concrete estimate for this infimum.
\end{rem}

\begin{proof}[Proof (of Lemma~\ref{lem:near_convex_comparable})]

    \begin{enumerate}
        \item Notice that the operator $\sum_{i=1}^{n}\lambda_iA_i$ is always positive definite. Thus, it suffices to show that if $\lambda<m=:\min_{i=1,\ldots,n}m(A_i)$, then the operator $P(\lambda):=\sum_{i=1}^{n}A_i-\lambda I$ is invertible. We can write
        \begin{equation*}
            P(\lambda)=\sum_{i=1}^{n}\lambda_i(A_i-\lambda I).
        \end{equation*}
        If $\lambda<m$, then each of the operators $A_i-\lambda I$ is positive definite. Hence, any convex combination of them is also positive definite.

        \item We only show that $C\leq(1+\e)D$, the estimate $D
        \leq(1+\e)C$ being symmetric. Clearly, it suffices to prove that $(1+\e)D-C$ is positive semidefinite. Set $c_i:=\frac{\mu_i-\lambda_i}{\e}$. We compute
        \begin{equation*}
            (1+\e)D-C=\sum_{i=1}^{n}[(1+\e)\mu_i-\lambda_i]A_i=\e\sum_{i=1}^{n}(\mu_{i}+c_i)A_i.
        \end{equation*}
        Noticing that $\sum_{i=1}^{n}\mu_{i}A_i$ is positive definite and
        \begin{equation*}
            \left\Vert\sum_{i=1}^{n}c_iA_i\right\Vert\leq\sum_{i=1}^{n}|c_i|\cdot\Vert A_i\Vert\leq \min_{i=1,\ldots,n}m(A_i)\leq m\left(\sum_{i=1}^{n}\mu_iA_i\right),
        \end{equation*}
        we obtain the required result.
    \end{enumerate}
\end{proof}

Let now $\delta>0$ be arbitrary. We explain that by choosing large enough $d$ we have $S^{\text{dy}}_{\wt{W}}<1+\delta$ as well as $S^{\text{dy}}_{\wt{W}^{-1}}<1+\delta$. Precisely, let us number the dyadic grandchildren of any $K\in\cD$ as $K_1,K_2,K_3,K_4$, in an arbitrary fashion. We pick
\begin{equation*}
    d\geq\frac{1}{\delta}\max\left(\max_{K\in\cD}\frac{\sum_{i=1}^{4}\Vert\langle W\rangle_{K_i}\Vert}{\min_{i=1,2,3,4}m(\langle W\rangle_{K_i})},\max_{K\in\cD}\frac{\sum_{i=1}^{4}\Vert\langle W^{-1}\rangle_{K_i}\Vert}{\min_{i=1,2,3,4}m(\langle W^{-1}\rangle_{K_i})}\right).
\end{equation*}
We note that this choice makes sense, since the dyadic martingale of $F$ terminates after finitely many steps. Let now $I\in\cD$ be arbitrary. Then, by construction it is clear that there is $K\in\cD$ such that
\begin{equation}
    \label{new_is_convex_combination_of_old}
    \langle\wt{F}\rangle_{I_{+}}=\sum_{i=1}^{4}\lambda_i\langle F\rangle_{K_{i}}\quad\text{and}\quad\langle\wt{F}\rangle_{I_{-}}=\sum_{i=1}^{4}\mu_i\langle F\rangle_{K_{i}}.
\end{equation}
Observe that relations \eqref{new_is_convex_combination_of_old} hold also for the components $W$ and $W^{-1}$ of $F$, respectively $\wt{W}$ and $\wt{W}^{-1}$ of $\wt{F}$. Thus, by Lemma~\ref{lem:near_convex_comparable} we immediately deduce $\langle \wt{W}\rangle_{I_{\pm}}\leq(1+\delta)\langle \wt{W}\rangle_{I_{\mp}}$ and $\langle \wt{W}^{-1}\rangle_{I_{\pm}}\leq(1+\delta)\langle \wt{W}^{-1}\rangle_{I_{\mp}}$.

Finally, let us note that $[\wt{W}]_{\mathbf{A}_2}^{\text{dy}}\leq 16Q$. Let $I\in\cD$ be arbitrary. Then, we can write $\langle \wt{F}\rangle_{I}$ as a convex combination $\langle \wt{F}\rangle_{I}=\sum_{i=1}^{4}\lambda_i\langle F\rangle_{K_i}$ for some $K\in\cD$. Thus, using \cite[Lemma 3.1]{DoPeTrVo} we can estimate
\begin{align*}
    \langle\wt{W}\rangle_{I}&\leq\sum_{i=1}^{4}\langle W\rangle_{K_i}=4\langle W\rangle_{K}\leq 4[W]_{\mathbf{A}_2}^{\text{dy}}\langle W^{-1}\rangle_{K}^{-1}=16Q\left(\sum_{i=1}^{4}\langle W^{-1}\rangle_{K_i}\right)^{-1}\\
    &\leq 16Q\left(\sum_{i=1}^{4}\lambda_i\langle W^{-1}\rangle_{K_i}\right)^{-1}=16Q\langle\wt{W}^{-1}\rangle_{I}^{-1}.
\end{align*}
One more application of \cite[Lemma 3.1]{DoPeTrVo} yields $[\wt{W}]_{\mathbf{A}_2}^{\text{dy}}\leq 16Q$.

\section{Remodeling and finalization of the construction}

As last step of our construction, we apply the remodeling transform described in \cite[Section 7]{DoPeTrVo} on the $\mathscr{X}$-valued function $\wt{F} = (\wt{W},\wt{V},\wt{\mathbf{f}},\wt{\mathbf{g}})$ constructed in Section~\ref{sec:small_step_examples} above. Since we do not need to perform any changes as compared to \cite[Section 7]{DoPeTrVo}, we omit the details and instead refer the reader to \cite[Section 7]{DoPeTrVo}. We denote the new function arising through this construction again by $\wt{F} = (\wt{W},\wt{V},\wt{\mathbf{f}},\wt{\mathbf{g}})$, abusing notation.

As explained in \cite[Section 8]{DoPeTrVo}, by choosing the frequency parameters of the remodeling transform appropriately, we can ensure that
\begin{equation*}
    |(\mathcal{H}\wt{\mathbf{f}},\wt{\mathbf{g}})_{L^2(I^0)}|\gtrsim Q^{3/2}\|\wt{\mathbf{f}}\|_{L^2(\wt{W})}\|\wt{\mathbf{g}}\|_{L^2(\wt{W}^{-1})}.
\end{equation*}
In the rest of this section we check that the remodeled weight $\wt{W}$ satisfies the necessary conditions regarding Muckenhoupt characteristics and smoothness.

\subsection{Muckenhoupt characteristic and smoothness}
    
\subsubsection{Passage from the dyadic matrix \texorpdfstring{$A_2$}{A2} condition to the classical one} The weight $\wt{W}$ of Section~\ref{sec:small_step_examples} satisfies $[\wt{W}]_{\mathbf{A}_2^{\mathrm{dy}}}\sim Q$. As shown in \cite[Subsection 9.1]{DoPeTrVo}, the weight arising after in addition the remodeling transform has been applied to $\wt{W}$, which through abuse of notation is still denoted by $\wt{W}$, satisfies $[\wt{W}]_{\mathbf{A}_2}\sim Q$.
	
\subsubsection{Passage from dyadic smoothness to strong dyadic smoothness} Let $\e>0$ be arbitrary and let $\delta>0$ with $(1+\delta)^3\leq1+\e$. As explained in Section~\ref{sec:small_step_examples}, we can ensure that the weight $\wt{W}$ on $[0,1)$ constructed there satisfies $S_{\wt{W}}^{\text{d}}<1+\delta$. Then, similarly to the scalar case \cite[Lemma 6.1]{KaTr}, we have that after applying remodeling, the remodeled weight, still denoted by $\wt{W}$, satisfies $S_{\wt{W}}^{\text{sd}}<1+\e$.

\bibliography{bibliography}

\begin{thebibliography}{LSSUT14}

\bibitem[ACUM15]{AnCrMo}
Theresa~C. Anderson, David Cruz-Uribe, and Kabe Moen.
\newblock {Logarithmic bump conditions for {Calderón}-{Zygmund} operators on spaces of homogeneous type}.
\newblock {\em Publicacions Matemàtiques}, 59(1):17 -- 43, 2015.

\bibitem[Bou83]{Bo}
Jean Bourgain.
\newblock {Some remarks on {Banach} spaces in which martingale difference sequences are unconditional}.
\newblock {\em Arkiv för Matematik}, 21(1–2):163--168, December 1983.

\bibitem[Buc93]{Bu}
Stephen~M. Buckley.
\newblock {Estimates for Operator Norms on Weighted Spaces and Reverse {Jensen} Inequalities}.
\newblock {\em Transactions of the American Mathematical Society}, 340(1):253, November 1993.

\bibitem[CUP00]{CrPe}
David Cruz-Uribe and Carlos Pérez.
\newblock {Two-weight, Weak-type Norm Inequalities for Fractional Integrals, {Calderón}-{Zygmund} Operators and Commutators}.
\newblock {\em Indiana University Mathematics Journal}, 49(2):697--721, 2000.

\bibitem[CURV14]{CrReVo}
David Cruz-Uribe, Alexander Reznikov, and Alexander Volberg.
\newblock {Logarithmic bump conditions and the two-weight boundedness of {Calderón}–{Zygmund} operators}.
\newblock {\em Advances in Mathematics}, 255:706--729, April 2014.

\bibitem[DPTV24]{DoPeTrVo}
Komla Domelevo, Stefanie Petermichl, Sergei Treil, and Alexander Volberg.
\newblock {The matrix {$A_2$} conjecture fails, i.e. $3/2>1$}.
\newblock February 2024.

\bibitem[GPTV01]{GiPoTrVo}
Thomas~A. Gillespie, Sandra Pott, Sergei Treil, and Alexander Volberg.
\newblock {Logarithmic growth for matrix martingale transforms}.
\newblock {\em Journal of the London Mathematical Society}, 64(3):624--636, December 2001.

\bibitem[HMW73]{HuMuWh}
Richard Hunt, Benjamin Muckenhoupt, and Richard Wheeden.
\newblock {Weighted norm inequalities for the conjugate function and {Hilbert} transform}.
\newblock {\em Transactions of the American Mathematical Society}, 176:227--227, 1973.

\bibitem[HN94]{HaNi}
Victor~P. Hanin and Nikolai~K. Nikolski, editors.
\newblock {\em {Linear and Complex Analysis Problem Book 3}}.
\newblock Springer Berlin Heidelberg, 1994.

\bibitem[Hyt18]{Hy}
Tuomas~P. Hytönen.
\newblock {The two-weight inequality for the {Hilbert} transform with general measures}.
\newblock {\em Proceedings of the London Mathematical Society}, 117(3):483--526, April 2018.

\bibitem[Kak]{Ka_PhD}
Spyridon Kakaroumpas.
\newblock {\emph{Sharp Weighted Estimates in Harmonic Analysis}}.
\newblock PhD Thesis, Brown University, 2020.

\bibitem[Kak22]{Ka}
Spyridon Kakaroumpas.
\newblock {Two-weight estimates for sparse square functions and the separated bump conjecture}.
\newblock {\em Transactions of the American Mathematical Society}, February 2022.

\bibitem[Kle20]{Klenke2020}
Achim Klenke.
\newblock {\em {Probability Theory: A Comprehensive Course}}.
\newblock Springer International Publishing, 2020.

\bibitem[KS24]{KS2024}
Spyridon Kakaroumpas and Odí Soler.
\newblock {Preimages under linear combinations of iterates of finite {Blaschke} products}.
\newblock {\em Analysis and Mathematical Physics}, 14(3), June 2024.

\bibitem[KT21]{KaTr}
Spyridon Kakaroumpas and Sergei Treil.
\newblock {{``Small step''} remodeling and counterexamples for weighted estimates with arbitrarily “smooth” weights}.
\newblock {\em Advances in Mathematics}, 376:107450, 2021.

\bibitem[Lac14]{LaII}
Michael~T. Lacey.
\newblock {Two-weight inequality for the {Hilbert} transform: A real variable characterization, II}.
\newblock {\em Duke Mathematical Journal}, 163(15), December 2014.

\bibitem[Lac16]{LaIV}
Michael~T. Lacey.
\newblock {On the Separated Bumps Conjecture for {Calderón}-{Zygmund} Operators}.
\newblock {\em Hokkaido Mathematical Journal}, 45(2), June 2016.

\bibitem[Lac17]{LaIII}
Michael~T. Lacey.
\newblock {\em {The Two Weight Inequality for the {Hilbert} Transform: A Primer}}, pages 11--84.
\newblock Springer International Publishing, 2017.

\bibitem[Ler13]{Le}
Andrei~K. Lerner.
\newblock {On an estimate of {Calderón}-{Zygmund} operators by dyadic positive operators}.
\newblock {\em Journal d’Analyse Mathématique}, 121(1):141--161, October 2013.

\bibitem[LSSUT14]{LaI}
Michael~T. Lacey, Eric~T. Sawyer, Chun-Yen Shen, and Ignacio Uriarte-Tuero.
\newblock {Two-weight inequality for the {Hilbert} transform: A real variable characterization, I}.
\newblock {\em Duke Mathematical Journal}, 163(15), December 2014.

\bibitem[Mak89]{Makarov1990}
Nikolai~G. Makarov.
\newblock {Probability methods in the theory of conformal mappings}.
\newblock {\em Algebra i Analiz}, 1(1):3--59, 1989.

\bibitem[Naz]{Na}
Fedor Nazarov.
\newblock {\emph{A counterexample to {Sarason}'s conjecture}}.
\newblock Unpublished manuscript, available at \url{https://users.math.msu.edu/users/fedja/prepr.html}.

\bibitem[Neu83]{Ne}
Christoph~J. Neugebauer.
\newblock {Inserting {$A_p$}-Weights}.
\newblock {\em Proceedings of the American Mathematical Society}, 87(4):644, April 1983.

\bibitem[NPTV17]{NaPeTrVo}
Fedor Nazarov, Stefanie Petermichl, Sergei Treil, and Alexander Volberg.
\newblock {Convex body domination and weighted estimates with matrix weights}.
\newblock {\em Advances in Mathematics}, 318:279--306, October 2017.

\bibitem[NRTV13]{NaReTrVo}
Fedor Nazarov, Alexander Reznikov, Sergei Treil, and Alexander Volberg.
\newblock {A {Bellman} function proof of the {$L^2$} bump conjecture}.
\newblock {\em Journal d’Analyse Mathématique}, 121(1):255--277, October 2013.

\bibitem[Pet07]{Pe}
Stefanie Petermichl.
\newblock {The sharp bound for the {Hilbert} transform on weighted {Lebesgue} spaces in terms of the classical {$A_p$} characteristic}.
\newblock {\em American Journal of Mathematics}, 129(5):1355--1375, October 2007.

\bibitem[PV02]{PeVo}
Stefanie Petermichl and Alexander Volberg.
\newblock {Heating of the {Ahlfors}-{Beurling} operator: weakly quasiregular maps on the plane are quasiregular}.
\newblock {\em Duke Mathematical Journal}, 112(2), April 2002.

\bibitem[PW02]{PeWit}
Stefanie Petermichl and Janine Wittwer.
\newblock {A sharp estimate for the weighted Hilbert transform via Bellman functions}.
\newblock {\em Michigan Mathematical Journal}, 50(1):71 -- 88, 2002.

\bibitem[RS17]{RaSp}
Robert Rahm and Scott Spencer.
\newblock {Entropy Bumps and another sufficient condition for the two-weight boundedness of sparse operators}.
\newblock {\em Israel Journal of Mathematics}, 223(1):197--204, November 2017.

\bibitem[TV97]{TrVo}
Sergei Treil and Alexander Volberg.
\newblock {Wavelets and the Angle between Past and Future}.
\newblock {\em Journal of Functional Analysis}, 143(2):269--308, February 1997.

\bibitem[TV16]{TrVo1}
Sergei Treil and Alexander Volberg.
\newblock {Entropy conditions in two weight inequalities for singular integral operators}.
\newblock {\em Advances in Mathematics}, 301:499--548, October 2016.

\end{thebibliography}
\bibliographystyle{alpha}

\end{document}